\documentclass[11pt,a4paper]{amsart} 
\usepackage{a4wide,ifthen,amssymb,color}
\usepackage{mathtools}
\mathtoolsset{showonlyrefs}
\usepackage[all]{xypic}
\usepackage{mathrsfs}
\let\spacecal=\mathscr
\usepackage[backref=page]{hyperref}

\hypersetup{
 colorlinks,
 citecolor=red,
 linkcolor=blue,
 urlcolor=Blue}

\newtheorem{thm}{Theorem}[section] 
\newtheorem{corol}[thm]{Corollary}
\newtheorem{lemma}[thm]{Lemma} 
\newtheorem{prop}[thm]{Proposition}
\newtheorem{defin}[thm]{Definition}
\theoremstyle{definition}

\theoremstyle{remark}
\newtheorem{remark}[thm]{Remark}
\newtheorem{example}[thm]{Example}
\numberwithin{equation}{section}

\newcommand\rk{\operatorname{rk}}
\newcommand\Id{\operatorname{Id}}
\newcommand\Ima{\operatorname{Im}}
\newcommand\Spec{\operatorname{Spec}}
\newcommand\Hom{\operatorname{Hom}}
\newcommand\iso{\kern.35em{\raise3pt\hbox{$\sim$}\kern-1.1em\to}\kern.3em}

\newcommand\Pic{\operatorname{Pic}}

\newcommand\rest[2]{#1_{\vert #2}}

\newcommand\F{{\mathcal F}}
\newcommand\Oc{{\mathcal O}}

\newcommand\M{{\mathcal M}}
\newcommand\Z{{\mathbb Z}}

\newcommand\U{{\spacecal U}}

\newcommand\Dcal{{\mathcal D}}
\newcommand\Nc{{\mathcal N}}

\newcommand\Lcl{{\mathcal L}}
\newcommand\Ec{{\mathcal E}}
\newcommand\Ac{{\mathcal A}}
\newcommand\Bc{{\mathcal B}}

\newcommand\Gc{{\mathcal G}}
\newcommand\Xcal{{\spacecal X}}
\newcommand\Ycal{{\spacecal Y}}
\newcommand\Sc{{\spacecal S}}
\newcommand\Tc{{\spacecal T}}
\newcommand\Zc{{\spacecal Z}}

\newcommand\Vc{{\mathcal V}}
\newcommand\Rcal{{\spacecal R}}

\newcommand\Jc{{\mathcal J}}
\newcommand\Cc{{\mathcal C}}

\newcommand\Der{{{\mathcal D}er}}
\newcommand\Ber{{{\mathcal B}er}}

\newcommand\Div{{{\mathcal D}iv}}

\newcommand\Xf{{\mathfrak X}}
\newcommand\Yf{{\mathfrak Y}}
\newcommand\Sf{{\mathfrak S}}

\newcommand\Df{{\mathcal D}}
\newcommand\Zf{{\mathcal Z}}

\newcommand\nf{{\mathfrak n}}

\newcommand\RR{{Ramond-Ramond}}

\newcommand{\rra}{\rightrightarrows} 

\begin{document}
\title[Supermoduli of SUSY curves with Ramond punctures]{THE SUPERMODULI OF SUSY CURVES \\[3pt]  WITH RAMOND PUNCTURES}
\author[U. Bruzzo]{Ugo  Bruzzo$^\P$}
\author[D. Hern\'andez Ruip\'erez]{Daniel Hern\'andez Ruip\'erez$^\ddag$}

\email{bruzzo@sissa.it}
\email{ruiperez@usal.es}

\address{\P\ SISSA (Scuola Internazionale Superiore di Studi Avanzati), Via Bonomea 265, 34136, Trie\-ste, Italia;
Departamento de Matem\'atica, Universidad Federal da Pa\-ra\'i\-ba,  Campus I, Jo\~ao Pessoa, PB, Brazil; 
INFN (Istituto Nazionale di Fisica Nucleare), Sezione di Trieste; 
IGAP (Institute for Geometry and Physics), Trieste; 
Arnold-Regge Center for Algebra, Geometry and Theoretical Physics, Torino
}
\address{\ddag\ Departamento de Matem\'aticas and IUFFYM (Instituto Universitario de F\'{\i}sica Fundamental \\ y Matem\'a\-ticas),  Universidad de Salamanca, Plaza
de la Merced 1-4, 37008 Salamanca, Spain; \\
Real Academia de Ciencias Exactas, F\'{i}sicas y Naturales, Spain}
\date{9 February 2020. Revised 31 May 2020} 

\thanks{\P\ Research partly supported by GNSAGA-INdAM and by the PRIN project ``Geometria delle variet\`a algebriche.''}
\thanks{\ddag\ Research partly supported by research projects  ``Espacios finitos y functores integrales'', MTM2017-86042-P,  (Ministerio of Econom\'{\i}a, Industria y Competitividad) and  ``STAMGAD, SA106G19'' (Junta de Castilla y Le\'on)}


\begin{abstract} We construct local and global moduli spaces of supersymmetric curves with Ramond-Ramond punctures. 
We assume that the underlying ordinary algebraic curves have a level $n$ structure and build these supermoduli spaces as algebraic superspaces, i.e.,   quotients of   \'etale equivalence relations between superschemes. 
 \end{abstract}
\maketitle 

\section{Introduction}
Supersymmetry requires the introduction of spaces with ``both bosonic and fermionic coordinates,''
as fermion fields and fermionic parameters are anticommuting objects.
After such spaces were formally introduced by physicists (see e.g.~\cite{SS}),
mathematicians immediately set
out to look for a sound notion of these structures,  and in the beginning different kinds of ``supermanifolds'' were  proposed (see \cite{BBH91} and references therein). It became soon apparent that the Berezin-Le\u\i tes approach \cite{BL75}, also developed by  Kostant \cite{Kost75} and Manin \cite{Ma86,Ma87,Ma88,Ma88-2}, closer in spirit to algebraic geometry, provides the most suitable notion of supermanifold, or graded manifolds as Kostant called them, and of supervarieties or superschemes. 

Super Riemann surfaces,  or supersymmetric (SUSY) curves, are an interesting chapter in this story. Their introduction 
was advocated by Friedan \cite{Fr86} (but see also \cite{BaSchw87,Ma86}) in connection with superstring theory. In the Polyakov approach to the bosonic string  \cite{Pol81}  the quantum scattering   amplitudes are computed in terms of  an integral over a compactification of the moduli space of (complex) algebraic curves,
 with punctures   representing the  asymptotic particle states. However, one of the effects of the compactification is that the Polyakov measure acquires poles at the boundary.   The fermions that one introduces in supersymmetric models  help to cancel the poles, and Friedan suggested that a compactification of a suitable moduli of punctured SUSY curves could be the right integration space for   superstrings.

There are various constructions of these supermoduli spaces, different in techniques and nature. There are analytic constructions of the supermoduli of SUSY curves of genus $g\ge 2$ as an orbifold (see, for instance, \cite {CraRab88}), and, among them, the seminal paper by LeBrun and Rothstein \cite{LeRoth88} on which many other constructions rely.  In \cite{DoHeSa97} a supermoduli space for SUSY curves with Neveu-Schwarz (NS) punctures was constructed as an (Artin) algebraic superspace, that is, the quotient of an \'etale equivalence relation of superschemes. To ensure the existence of a fine moduli for  ordinary curves, these constructions fix a level $n$ structure, with $n\ge 3$, on the curves.
Quite recently, Codogni and Viviani have given a construction of the supermoduli of SUSY curves as a Deligne-Mumford (DM) superstack \cite{CodViv17} without considering level $n$ structures. The   supermoduli built in   \cite{DoHeSa97} as an algebraic superspace is actually an explicit construction of a coarse moduli space for the supermoduli DM stack, that should exist once the representation result in \cite{KeelMori97} will have been extended to supergeometry. Deligne's manuscript letter \cite{Del87} contains all the necessary information to construct  global compactified moduli superspaces as   algebraic superstacks obtained from the DM stack of stable algebraic curves.

In the  last years, there has been an increasing  renewal of interest in supergeometry, in particular in the supermoduli of SUSY curves, with and without punctures. The most important  works that have recently brought new life to supergeometry are those of Witten \cite{Witten19} and Donagi-Witten \cite{DoW13,DoW15}.
Although no construction of the supermoduli spaces is offered there, those papers  contain  several results on the nonprojectedness of the supermoduli, which have important consequences for perturbative superstring theory, as this shows that one cannot integrate on the supermoduli by first integrating over the  fibres of a (nonexisting) projection to the ordinary moduli. 
Some earlier results about the projectedness of the supermoduli can be found  in \cite{FaRe90}.
Moreover,  substantial work has  been done about superperiods  and the generalization  of the Mumford formula  to the supermoduli of SUSY curves \cite{Witten19, Dir19, CodViv17, FKP19}.

This paper is devoted to the study of \RR\ supersymmetric curves (RR-SUSY curves, for short), i.e., loosely speaking,
SUSY curves that have Ramond-Ramond punctures supported along positive divisors. We study their geometry and construct their supermoduli.  To our knowledge, there is no other global construction of the supermoduli of SUSY curves with \RR\ punctures  for positive genus
available in the literature
(see the recent preprint \cite{Ott-Vor} for the genus zero case). 
We prove that the moduli functor of \emph{isomorphisms classes} of RR-SUSY curves is representable by an (Artin) algebraic superspace (this  provides a more detailed information than just constructing the moduli space as a Deligne-Mumford stack). In the complex case, algebraic spaces, that is, the quotients of \'etale equivalence relations of  schemes  (locally of finite type) \cite{Ar71,Knut71} are precisely those sheaves whose underlying analytic space is a Mo\u\i\v sezon space, namely, an analytic space whose field of meromorphic functions has transcendence degree equal to the dimension of the space \cite{Moi69}. In the ``super'' case one defines algebraic superspaces as quotients of \'etale equivalence relations of superschemes.

The paper is organized as follows. In Section \ref{s:super}, to fix the notation
and for the sake of a reasonable self-completeness, 
 we provide a summary of well-known basic facts about superschemes and their morphisms.  The material here can be easily found in the literature, though sometimes the definitions and results are scattered through many different sources. Section \ref{s:RRsusy} is devoted to the geometry of supercurves, understood as superschemes of dimension $(1,1)$,  of positive superdivisors and of SUSY curves with \RR\ punctures along a positive supervisor. According to the physical interpretations of punctures, we assume that the positive superdivisors along which the \RR\ punctures are supported are reduced, or, in the case of families, that the superdivisors are nonramified over the base. SUSY curves with \RR\ punctures are not SUSY curves, since they are equipped with an odd distribution which is superconformal only in the complement of the divisors. For a single RR-SUSY curve, around a prime divisor of the puncture the conformal distribution is locally generated by a vector field
$$
D=\frac{\partial}{\partial \theta}+z\theta \frac{\partial}{\partial z}.
$$
It is worth noticing that, contrary to what happens for SUSY curves, for a family $\Xcal\to\Sc$ of RR-SUSY curves this   is true only after taking an \'etale covering of the base $\Sc$ (Proposition \ref{p:localsplitRR}).
The infinitesimal deformations of an RR-SUSY curve $\Xcal\to S$ (over an ordinary scheme) are computed in Corollary \ref{cor:h1dimrel}. They are needed to give the local   fermionic structure of the supermoduli.

Section \ref{s:modRR} is devoted to the global construction of the supermoduli of RR-SUSY curves of genus $g$ with $\nf_R$ \RR\ punctures. 
We assume $g\geq 2$ and   fix a level $n$ structure with $n\geq 3$.
The first step of the construction is the determination of the underlying ordinary space to the supermoduli. This happens to be the moduli scheme $M_{RR}$ of RR-spin curves, that is, of curves with a line bundle whose square is the canonical sheaf twisted by the inverse of the ideal sheaf of the puncture (Theorem \ref{thm:RRfunct}). It is a fine moduli space, in the sense that it represents the sheafification of the functor of RR-spin curves, {though}, as the latter is not a sheaf, $M_{RR}$ does not carry a universal RR-spin curve. However, it has a universal ``RR-SUSY curve class,'' namely, there is a universal curve $X_{RR}\to M_{RR}$ with a relative positive divisor $Z_{RR}$ of degree $\nf_R$ and a class $\Upsilon$ in $\Pic(X_{RR}/M_{RR})$ such that $\Upsilon^2=[\kappa_{X_{RR}/M_{RR}}\otimes\Oc(Z_{RR})]$ in the relative Picard group, but $\Upsilon$ may fail to be the class of a line bundle on $X_{RR}$. We prove that there is an \'etale covering $U\to M_{RR}$ such that the pull-back $\Upsilon_U$ is the class of a line bundle $\Lcl$ on $X_U$ verifying $\Lcl^2\simeq \kappa_{X_U/U}\otimes\Oc(Z_U)$. 

We can then move on  to the second step of the construction: the determination of the local fermionic structure of the supermoduli $\M_{RR}$. If it exists, we see that locally its structure sheaf is the exterior algebra of the sheaf of odd infinitesimal deformations. Since we have already computed the latter, we get a local candidate for the supermoduli, namely, the superscheme $\U=(U, \bigwedge R^1\pi_{U\ast} (\kappa (Z)^{1/2}\otimes\kappa))$, where $\kappa=\kappa_{X_U/U}$ and $Z=Z_U$ (Definition \ref{def:localmoduliRR}). To complete the global construction we extend to RR-SUSY curves two results of  LeBrun and Rothstein \cite{LeRoth88} (our theorems \ref{LBR1-RR} and \ref{LBR2-RR}). The key point for the extension is Corollary \ref{cor:h0dimrel}, which reflects the fact that, under the hypotheses  we have made,
 RR-SUSY curves have a finite number of automorphisms.

The supermoduli space
turns out to be an algebraic superspace of dimension $(3g-3+\nf_R,2g-2+\nf_R/2)$. This is the content of Theorem \ref{thm:globlamoduiRR}. Moreover in Proposition \ref{p:universalcurveRR}  the ``universal RR-SUSY curve class'' $\Xf_{\nf_{NS},\nf_{R}}\to \M_{RR}$ over the supermoduli is also proven to have the structure of an algebraic superspace.

Finally, in Section \ref{s:allpunctures} we combine our construction for the supermoduli of SUSY curves with \RR\ punctures with the supermoduli for SUSY curves with Neveu-Schwarz punctures to obtain the algebraic superspace which is the supermoduli for SUSY curves with booth types of punctures.

{\bf Acknowledgements.} We thank the participants in the ``Workshop on Supermoduli'' which took place
at the Institute for Geometry and Physics in Trieste on September 23 to 26, 2019, for the nice atmosphere and the fruitful interchanges which occurred. We also thank Ron Donagi for stimulating discussions and   for sharing ideas, and Alexander Polishchuk for pointing out a small mistake in a previous version. We thank the referee for comments that allowed us to give a more precise statement
of Proposition \ref{p:caractberRR}.

\bigskip
\section{Superschemes and morphisms}\label{s:super}

In this section we recall  some basic definitions in supergeometry. We fix an algebraically closed field $k$  of characteristic different from 2; however, from some point on we shall assume $k=\mathbb C$.

\subsection{Superschemes}
We start by recalling the notion of superscheme.

\begin{defin}
A   locally ringed superspace is a pair $\Xcal=(X,\Oc_{\Xcal})$, where $X$ is a topological space, and $\Oc_{\Xcal}$ is a sheaf of $\Z_2$-graded commutative algebras over $k$ such that every stalk $\Oc_{\Xcal,x}$ is a local ring for every point $x\in X$. 
\end{defin}

There is a decomposition $\Oc_{\Xcal}=\Oc_{\Xcal,0}\oplus \Oc_{\Xcal,1}$ into the part of degree 0 (or even part) and the part of degree 1 (or odd part). As usual, by ``graded commutative'' we understand that   $a\cdot b=(-1)^{\vert a\vert\,\vert b\vert}b\cdot a$ for homogenous local sections $a$, $b$, where $\vert\ \vert$ denotes the $\Z_2$-grading.

We refer to any text on supergeometry for the usual definitions of ($\Z_2$-graded commutative) sheaf, homogeneous morphism of sheaves, (graded commutative) tensor product, (graded commutative) derivation, etc. (\cite{BBH91,Ma88}).
Morphisms of  locally ringed superspaces are defined in a similar way as morphisms of ordinary locally ringed spaces:
\begin{defin} \label{def:mor}
A morphism of locally ringed superspaces is a pair $(f,f^\sharp)$, where $f\colon X \to Y$ is  a continuous map, and  $f^\sharp\colon \Oc_{\Ycal}\to f_\ast\Oc_{\Xcal}$ is a homogeneous morphism of graded commutative sheaves, such that for every point $x\in X$, the induced morphism of graded commutative rings $\Oc_{\Ycal,f(x)}\to \Oc_{\Xcal,x}$ is local, that is, it maps the maximal ideal of $\Oc_{\Ycal,f(x)}$ into the maximal ideal of $\Oc_{\Xcal,x}$. When no confusion can arise we denote a morphism of of locally ringed superspaces simply by $f\colon \Xcal \to \Ycal$.
\end{defin}

If $\Xcal=(X,\Oc_{\Xcal})$ is a locally ringed superspace and $U\subset X$ is an open subset, we say that $(U, \rest{{\Oc_{\Xcal}}}U)$ is an \emph{open locally ringed sub-superspace} of $\Xcal$.

\begin{defin} A morphism $f\colon \Xcal\to \Ycal$ of locally ringed superspaces is an \emph{open immersion} if is induces an isomorphism of $\Xcal$ with an open locally ringed sub-superspace of $\Ycal$. 
It is a \emph{closed immersion} if the underlying continuous map $X\to Y$ is a homeomorphism onto a closed subset of $Y$ and  $f^\sharp\colon \Oc_{\Ycal}\to f_\ast\Oc_{\Xcal}$ is surjective.
It  is an \emph{immersion} if it is the composition $f=i\circ g$ of a closed immersion $g$ with an open immersion $i$.
\end{defin}

Given a  locally ringed superspace $\Xcal=(X,\Oc_{\Xcal})$, we can consider the homogeneous ideal  $\Jc=(\Oc_{\Xcal})_1^2\oplus(\Oc_{\Xcal})_1$ generated by the odd elements. Then, $\Oc_X:=\Oc_{\Xcal}/\Jc$ is a (purely even) sheaf of $k$-algebras. We say that the locally ringed space $X=(X,\Oc_X)$ is the \emph{ordinary locally ringed space underlying} $\Xcal$.
A morphism $f\colon \Xcal \to \Ycal$ of locally ringed superspaces induces a morphism $f\colon X \to Y$ of the underlying locally ringed spaces, so that $\Xcal=(X,\Oc_{\Xcal}) \mapsto X=(X,\Oc_X)$ is a functor.

The sheaves $Gr^j(\Oc_{\Xcal})=\Jc^j/\Jc^{j+1}$ are annihilated by $\Jc$ so that they are $\Oc_X$-modules. Then we   can   consider the sheaf of $\Oc_X$-modules 
$$
Gr (\Oc_{\Xcal})=\bigoplus_{j\ge 0} Gr^j(\Oc_{\Xcal})=\bigoplus_{j\ge 0} \Jc^j/\Jc^{j+1}\,.
$$
which comes with a natural $\Z_2$ grading.

\begin{defin}\label{def:superscheme}
A superscheme is a locally ringed superspace $\Xcal=(X,\Oc_{\Xcal})$ such that 
\begin{enumerate}
\item the underlying ordinary locally ringed space $X=(X,\Oc_X)$ is a scheme, which we always assume to be locally of finite type over
$k$;
\item $Gr (\Oc_{\Xcal})$ is coherent as an $\Oc_X$-module;
\item  $\Oc_{\Xcal}$ is locally isomorphic, as a sheaf of $\Z_2$-graded commutative algebras, with $Gr (\Oc_{\Xcal})$,  compatibly with the projection $\Oc_{\Xcal}\to \Oc_X$.
\end{enumerate}
\end{defin}

Note that the second condition is equivalent to the fact that the sheaves $Gr^j(\Oc_{\Xcal})=\Jc^j/\Jc^{j+1}$, and in particular $\Ec=Gr^1(\Oc_{\Xcal})$, are coherent and that there are locally only finitely many of them. This implies that for every point $x\in X$ there exists a neighbourhood $U$ and an integer $n=n(x)$ such that $\rest{\Jc}U^{n+1}=0$.

\begin{example}
If $\Ec$ is a coherent sheaf on a scheme $X$, the exterior algebra $\bigwedge_{\Oc_X}(\Ec)$ defines a superscheme $(X, \bigwedge_{\Oc_X}(\Ec))$. In this case, there is an isomorphism of algebras $\bigwedge_{\Oc_X}(\Ec)\iso Gr(\Oc_{\Xcal})$.
\end{example}

Note that, in general, for every superscheme $\Xcal$ there is always a surjective morphism of $\Oc_X$-graded sheaves
$$
\textstyle{\bigwedge_{\Oc_X}(\Ec)} \to Gr(\Oc_{\Xcal})\to 0\,.
$$

A superscheme $\Xcal$ is said to be \emph{affine} if $X$ is affine. In this case, if $A=\Gamma(X,\Oc_X)$ and $\Ac=\Gamma(X,\Oc_\Xcal)$, one has $X=\Spec A$, and $\Oc_{\Xcal}$ is the sheaf associated to the presheaf of localizations of the graded commutative ring $\Ac$.

\begin{defin}\label{def:supperscheme}
A superscheme $\Xcal=(X,\Oc_{\Xcal})$ is \emph{split} if it is isomorphic to the superscheme $(X, \bigwedge_{\Oc_X}(\Ec))$ (compatibly with the projection $\Oc_{\Xcal}\to\Oc_X$) and moreover $\Ec$ is \emph{locally free} as an $\Oc_X$-module.
A superscheme is \emph{locally split} if it can be covered by open sub-superschemes that are split.
\end{defin}

\begin{defin} A locally split superscheme $\Xcal=(X,\Oc_{\Xcal})$ has dimension $(m,n)$ if the scheme $X$ has dimension $m$ and the rank of $\Ec$ is $n$.
\end{defin}

\begin{remark}In most applications, all   superschemes considered are locally split, or locally fermionic trivial as they are also called by some authors. From Subsection \ref{supercurves}  on, all   superschemes will be assumed to be locally split. 
\end{remark}

\begin{prop}\label{prop:curvesplit}
A locally split superscheme $\Xcal=(X,\Oc_{\Xcal})$ of odd dimension $n=1$ is canonically split.
\end{prop}
\begin{proof}
The epimorphism $\iota_\sharp\colon \Oc_\Xcal\to \Oc_X$ induces an isomorphism $(\Oc_\Xcal)_0\iso \Oc_X$ between the even part of $\Oc_\Xcal$ and $\Oc_X$. Since $n=1$, the sheaf $\Lcl=\Jc/\Jc^2=\Jc$ is an invertible sheaf on $X$ and $(\Oc_\Xcal)_1=\Lcl^\Pi$, where $\Pi$ denotes the parity change fucntor. One then has
$$
\Oc_\Xcal\simeq \Oc_X\oplus\Lcl^\Pi\iso \textstyle{\bigwedge_{\Oc_X}}(\Lcl)\,.
$$
\end{proof}

We shall only consider sheaves of $\Oc_{\Xcal}$-modules  which are $\Z_2$-graded. Morphisms of sheaves of (graded) $\Oc_{\Xcal}$-modules are supposed to be homogeneous of degree $0$, unless some other condition is explicitly stated.

A free $\Oc_{\Xcal}$-module of rank $(p,q)$ is a graded $\Oc_{\Xcal}$-module that is 
isomorphic to  the sheaf $\Oc_{\Xcal}^{(p,q)}=\Oc_{\Xcal}^p\oplus(\Oc_\Xcal^q)^\Pi$ freely generated by $p$ even and $q$ odd elements. Similar definitions can be given for $\Oc_{\Xcal}^{(I,J)}$ for arbitrary sets $I$, $J$. A sheaf $\Nc$ of $\Oc_{\Xcal}$-modules  is of finite type if there exists a (homogeneous of degree $0$) epimorphism $\Oc_{\Xcal}^{(p,q)}\to \Nc\to 0$ for some pair $(p,q)$.

The usual notions of quasi-coherent and coherent sheaves apply also here. Note that, due to the hypotheses in Definition \ref{def:superscheme}, coherent sheaves are locally   finite type quasi-coherent sheaves.

For every superscheme $\Xcal$, the surjection $\Oc_{\Xcal} \to \Oc_X=\Oc_{\Xcal}/\Jc$ induces a closed immersion
$$
i\colon X \hookrightarrow \Xcal
$$ 
of superschemes, defined by the (graded) ideal $\Jc$. We say that $\Xcal$ is \emph{projected} if there is a morphism $\rho\colon \Xcal \to X$ such that $\rho\circ i$ is the identity. In particular, split superschemes are projected.

The following result will be useful.
\begin{lemma}\label{lem:vanishX}
 Let $\Nc$ be a quasi-coherent $\Oc_{\Xcal}$-module. If $i^\ast\Nc=0$, then $\Nc=0$.
\end{lemma}
\begin{proof} Since $0=i^\ast\Nc=\Nc\otimes_{\Oc_{\Xcal}}\Oc_X$, one has $\Nc=\Jc\cdot\Nc$, and then $\Nc=\Jc^{n+1}\cdot\Nc$ for every $n$. Then $\Nc=0$ because $X$ can be covered by open subschemes $U$ such that there exist integers $n(U)$ verifying $\rest{\Jc}{U}^{n(U)+1}=0$.
\end{proof}

Given morphisms $f\colon \Xcal\to \Sc$, $g\colon \Zc \to \Sc$, a \emph{fibre product} $f\times g\colon \Xcal\times_\Sc \Zc \to \Sc$ exists, together with two projections $p_1\colon \Xcal\times_\Sc \Zc \to \Xcal$, $p_2\colon \Xcal\times_\Sc \Zc\to \Zc$. For affine superschemes given by graded commutative algebras $\Ac=\Gamma(X,\Oc_\Xcal)$,  $\Bc=\Gamma(S,\Oc_\Sc)$, $\Cc=\Gamma(Z,\Oc_\Zc)$, $\Xcal\times_\Sc \Zc$ is the affine superscheme defined by $\Ac\otimes_\Bc \Cc$ with its natural $\Z_2$-grading.

For every morphism $f\colon \Xcal\to \Sc$ of superschemes there exists a diagonal morphism $\delta_f\colon \Xcal \hookrightarrow \Xcal\times_\Sc \Xcal$. If $V\subseteq S$ is an affine open subscheme of $S$ and $U\subseteq f^{-1}(V)$ is an affine subscheme of $X$, and $\Vc$, $\U$ are the induced open sub-superschemes, the diagonal of $f\colon \U\to \Vc$ is defined by the kernel $\Delta_{\U/\Vc}$ of the product $\Oc_{\Xcal}(U)\otimes_{\Oc_{\Sc}(V)}\Oc_{\Xcal}(U) \to \Oc_{\Xcal}(U)$. In this way one defines a sheaf $\Delta_f$ of graded ideals on $\Xcal\times_\Sc \Xcal$, called the \emph{diagonal ideal} of $f\colon \Xcal\to \Sc$. 
The quotient $\Delta_f/\Delta_f^2$ is supported on the diagonal, i.e., $\Delta_f/\Delta_f^2\simeq \delta_{f,\ast}\delta_f^\ast (\Delta_f/\Delta_f^2)$.

\begin{defin}
The relative cotangent sheaf to $f\colon \Xcal\to \Sc$ is the sheaf
$$
\Omega_{\Xcal/\Sc}=\delta_f^\ast(\Delta_f/\Delta_f^2)\,.
$$
\end{defin}

The local sections of $\Omega_{\Xcal/\Sc}$ are called ($\Z_2$-graded) relative differentials. On affine open sub-superschemes they are generated by $df$ where $f$ is an homogenous section of $\Oc_{\Xcal}$.
As for ordinary schemes, there is a derivation (of degree $0$) over $\Oc_{\Sc}$:
\begin{equation}
\begin{aligned}
\Oc_{\Xcal} & \xrightarrow{d} \Omega_{\Xcal/\Sc} \\
f & \mapsto df=[f\otimes 1-1\otimes f]
\end{aligned}
\end{equation}
(where square brackets denote restriction to the diagonal) and  $\Omega_{\Xcal/\Sc}$ has a universal property with respect to   graded derivations. In particular, there is an isomorphism of (graded) $\Oc_{\Xcal}$-modules
$$
\Omega_{\Xcal/\Sc}^\ast \iso \Der_{\Oc_{\Sc}}(\Oc_{\Xcal})\,.
$$
This sheaf is called the \emph{relative tangent sheaf} of $f\colon \Xcal\to \Sc$ and  is denoted  $\Theta_{\Xcal/\Sc}$.

If one considers the structural morphism $\Xcal \to\Spec k$ of a $k$-superscheme $\Xcal$  the absolute notions are obtained. Then, the \emph{cotangent sheaf} of a superscheme $\Xcal$ is the sheaf $\Omega_{\Xcal}$ of ($\Z_2$-graded) differentials  of  $\Xcal \to\Spec k$, and $\Theta_\Xcal=\Der(\Oc_{\Xcal})\iso \Omega_{\Xcal}^\ast$ is the \emph{tangent sheaf}.
For every point $x\in X$ (closed or not), the \emph{tangent space to $\Xcal$ at $x$} is the graded vector space
$$
\Theta_{\Xcal,x}=\Der_k(\Oc_{\Xcal,x}, k(x))\,,
$$
where $k(x)$ is the residue field of the local ring $\Oc_{X,c}$, that is, $\Theta_{\Xcal,x}$ is the fibre of the tangent sheaf at $x$. One has $\Theta_{\Xcal,x}=\Der_k(\Oc_{\Xcal,x}, k)$ when $x$ is a closed point.\footnote{There is a notational inconsistency here, as $\Theta_{\Xcal,x}$ is not the stalk $\Der(\Oc_{\Xcal})_x= \Der(\Oc_{\Xcal,x})$ but rather its fibre $\Der(\Oc_{\Xcal,x})\otimes_{\Oc_{X,x}}k(x)$. We keep this inconsistency for historical reasons.}

Let us consider the superscheme $\Spec k[\epsilon_0,\epsilon_1]$, defined as $(\Spec k, k[\epsilon_0,\epsilon_1])$ where $\epsilon_i$ has parity $i$ and
$\epsilon_0^2=\epsilon_1^2=\epsilon_0\epsilon_1=0$.

If $x$ is a closed point, a ``super tangent vector'' $v_x\in \Theta_{\Xcal,x}$ defines a morphism $\phi(v_x)\colon \Oc_{\Xcal,x} \to k[\epsilon_0,\epsilon_1]$ of graded $k$-algebras, given by $\phi(v_x)(a)=\bar a +(v_x)_0(a)\epsilon_0+(v_x)_1(a)\epsilon_1$. Conversely, any morphism $\phi\colon \Oc_{\Xcal,x} \to k[\epsilon_0,\epsilon_1]$ of graded $k$-algebras defines a supertangent vector $v_x$ such that $\phi=\phi(v_x)$. That is, there is an isomorphism
\begin{equation}\label{stvectors}
\Hom_{x}(\Spec k[\epsilon_0,\epsilon_1], \Xcal)\iso \Theta_{\Xcal,x}
\end{equation}
of graded $k$-vector spaces, where the first member denotes the morphisms $\Spec k[\epsilon_0,\epsilon_1]\to \Xcal$ centred on $x$, that is, which map the single point of $\Spec k[\epsilon_0,\epsilon_1]$ to $x$.

Let us write
\begin{equation}\label{eq:epsilon}
\Xcal[\epsilon_0,\epsilon_1]=\Xcal\times \Spec k[\epsilon_0,\epsilon_1]\,,
\end{equation}
for every supescheme $\Xcal$. Proceeding as above one can describe the sections of the restriction $\Theta_{\Xcal\vert X}$ of the tangent sheaf to the underlying scheme. Namely:
\begin{prop}\label{p:tangentdefor}
 For every affine open subscheme $U\hookrightarrow X$, the space
$\Gamma(U, \Theta_{\Xcal\vert X})=\Der(\Oc_\Xcal(U),\Oc_X(U))$ is identified with the subspace of the elements in $\Hom(U[\epsilon_0,\epsilon_1],\Xcal)$ that are extensions of the natural immersion $U\hookrightarrow\Xcal$.

More generally, if $f\colon\Ycal\to\Xcal$ is a morphism of superschemes, for every affine open subscheme $V\hookrightarrow Y$ the space $\Gamma(V, f^\ast\Theta_\Xcal)$ of sections of the pullback of the tangent sheaf is identified with the subspace of   elements in $\Hom(\Ycal_V[\epsilon_0,\epsilon_1],\Xcal)$ that are extensions of $\Ycal_V\hookrightarrow\Ycal\xrightarrow{f}\Xcal$.
\qed
\end{prop}
When $\Ycal=\Spec k$ is a single closed point of $\Xcal$, we recover Equation \eqref{stvectors}.

\begin{defin}\label{def:smooth}
A superscheme $\Xcal=(X,\Oc_{\Xcal})$ is smooth of dimension $(m,n)$ if
\begin{enumerate}
\item every irreducible component of $X$ has dimension $m$, and
\item for every point $x\in X$ (not necessarily closed), the stalk $\Omega_{\Xcal,x}$ of the cotangent sheaf at $x$ is a free $\Oc_{\Xcal,x}$-module of rank $(m,n)$.
\end{enumerate}
\end{defin}

\begin{prop}
Let  $\Xcal=(X,\Oc_{\Xcal})$ be a superscheme,
and $\Ec=Gr^1(\Oc_{\Xcal})= \Jc/\Jc^2$. Then $\Xcal$ is smooth of dimension $(m,n)$ if and only if
\begin{enumerate}
\item it is locally split and $n=\rk \Ec$, and 
\item 
$X$ is a smooth scheme of dimension $m$.
\end{enumerate}
Then, if $\Xcal$ is smooth, for  every closed point $x\in X$ there are \emph{graded local coordinates}, that is, $m$ even functions $(z_1,\dots,z_m)$ in the maximal ideal $\mathfrak m_x$ of $\Oc_{X,x}$ and $n$  odd functions $(\theta_1,\dots,\theta_n)$ generating $\Ec_x$, such that $(dz_1,\dots,dz_m,d\theta_1,\dots,d\theta_n)$ is a   basis for $\Omega_{\Xcal,x}$.
\end{prop}
\begin{proof} Assume that the two conditions hold. Since the question is local, we can suppose that $\Oc_{\Xcal}=\Oc_X[\theta_1,\dots,\theta_n]$ where the $\theta_i$'s are odd and generate $\Ec$. Then $\Omega_{\Xcal/X}\simeq \Oc_X[d\theta_1,\dots,d\theta_n]$ and the exact sequence of $\Oc_{\Xcal}$-modules
$$
\Omega_X\otimes_{\Oc_X}\Oc_{\Xcal}\to \Omega_{\Xcal}\to \Omega_{\Xcal/X}\to 0\,,
$$
is split. It follows that
$$
\Omega_{\Xcal}\iso (\Omega_X\otimes_{\Oc_X}\Oc_{\Xcal}) \oplus \Oc_X[d\theta_1,\dots,d\theta_n]\,.
$$ 
Since $X$ is smooth of dimension $m$, $\Omega_{X,x}$ is free of rank $m$ for every point $x$, and then the two conditions in Definition \ref{def:smooth} are fulfilled.

For the converse, the question is again local, so we can also assume that $X$ is the spectrum of a local ring, which we still denote by $\Oc_X$. The exact sequence induced by $i\colon X\hookrightarrow \Xcal$ gives an exact sequence
$$
\Ec=\Jc/\Jc^2\xrightarrow{\delta} \Omega_{\Xcal}\otimes_{\Oc_{\Xcal}}\Oc_X \to \Omega_X \to 0
$$
where $\delta(\theta)$ is the class modulo $\Jc$ of $d\theta$. Decomposition into even and odd parts gives
$$
\begin{aligned}
&\Omega_{\Xcal,0}\otimes_{\Oc_{\Xcal}}\Oc_X\iso \Omega_X\\
\Ec\iso  &\Omega_{\Xcal,1}\otimes_{\Oc_{\Xcal}}\Oc_X\,,
\end{aligned}
$$
where the second equality  holds as $\Omega_{\Xcal}$ is free. 
This proves that $\Ec$ and $\Omega_X$ are free $\Oc_X$-modules of rank $n$ and $m$. In particular $X$ is smooth as an ordinary scheme (see, for instance, \cite{Hart77}). We only have to prove that the natural epimorphism  $\bigwedge_{\Oc_X}\Ec\to \Oc_{\Xcal}=Gr (\Oc_{\Xcal})$ is an isomorphism. We proceed by induction on $n$, the case $n=0$ being trivial. Take $(z_1,\dots,z_m)$ in the maximal ideal $\mathfrak m$ or $\Oc_X$ and  $(\theta_1,\dots,\theta_n)$ in $\Ec$ such that $(dz_1,\dots,dz_m,d\theta_1,\dots,d\theta_n)$ is a  basis of $\Omega_{\Xcal}$. If $\Sc$ is the closed sub-superscheme defined by $\theta_m$, $\Omega_{\Sc}$ is a free $\Oc_{\Sc}/(\theta_n)$ module generated by $(dz_1,\dots,dz_m,d\theta_1,\dots,d\theta_{n-1})$; if $\Ec'=\Ec/\langle\theta_n\rangle$, one has $\bigwedge_{\Oc_X}\Ec'\iso \Oc_{\Sc}$, by induction. Now, if  an $i$-form $\omega$ is in the kernel of $\bigwedge_{\Oc_X}\Ec\to \Oc_{\Xcal}=Gr (\Oc_{\Xcal})$, reducing modulo $\theta_n$ and applying the above isomorphism we see that $\omega=\sum a_{j_1\dots j_{i-1}}\theta_{j_1}\wedge\dots\wedge\theta_{j_{i-1}} \wedge\theta_n$, where the sum runs on  $1\leq j_1<\dots<j_{i-1}\leq n-1$ and the coefficients are in $\Oc_X$. Taking the differentials on $0=\sum a_{j_1\dots,j_{i-1}}\theta_{j_1}\cdot\dots\cdot\theta_{j_{i-1}} \cdot\theta_n$ and using that $(dz_1,\dots,dz_m,d\theta_1,\dots,d\theta_n)$ is a  basis of $\Omega_{\Xcal}$, we obtain that the coefficient of $d\theta_n$ is zero, that is, $0=\sum a_{j_1\dots,j_{i-1}}\theta_{j_1}\cdot\dots\cdot\theta_{j_{i-1}}$. But then, $\sum a_{j_1\dots j_{i-1}}\theta_{j_1}\wedge\dots\wedge\theta_{j_{i-1}}$ is in the kernel of $\bigwedge_{\Oc_X}\Ec'\iso \Oc_{\Sc}$, so that it is is zero as well and $\omega=0$.
\end{proof}

If $\Xcal$ is a smooth superschemes, it is customary to denote by $\Omega_+\Xcal$ and $\Omega_-\Xcal$ the even an odd parts of the restriction $\Omega_{\Xcal\vert X}=i^\ast\Omega_\Xcal$ of $\Omega_\Xcal$ to the underlying scheme, and similarly, by $\Theta_+\Xcal$ and $\Theta_-\Xcal$ the even and odd parts of the restriction of the tangent sheaf $\Theta_\Xcal$, and to refer to them as   the \emph{even or the odd cotangent or tangent sheaves} to $\Xcal$. The above proof shows that 
\begin{equation}\label{eq:bostangent}
\begin{aligned}
&\Omega_+\Xcal\iso \Omega_X\,,\quad &&\Ec^\Pi\iso\Omega_-\Xcal
\\
&\Theta_X\iso \Theta_+\Xcal\,,\quad &&\Theta_-\Xcal\iso\Ec^{\ast\Pi}\,,
\end{aligned}
\end{equation}
where we use the functor $\Pi$ to keep track of the parity.
This means that $\Xcal$ \emph{is locally determined by either its odd tangent or cotangent sheaf}. We shall exploit this property in the determination of the local fermionic structure of the supermoduli.

\subsection{Morphisms of superschemes}

We only consider morphisms  $f\colon \Xcal \to \Sc$ of superschemes that are \emph{locally of finite type}. Since our superschemes are locally of finite type as well, all morphisms are automatically locally of finite presentation.

Many of the notions of the different types of morphisms are defined mimicking the corresponding definitions for schemes, and no further explanation is needed.

The \emph{fibre} of a morphism $f\colon \Xcal \to \Sc$ of superschemes over a point $s\in S$ is the superscheme $\Xcal_s=\Xcal\times_\Sc \Spec \kappa(s)$, where $\kappa(s)$ is the residue field of $s\in S$.

Whenever one has a morphism of superschemes $f\colon \Xcal \to \Sc$, we also say that it is a \emph{relative superscheme} or that $\Xcal$ is \emph{superscheme over $\Sc$}.


%

\begin{prop}\label{p:flat}
Let $f\colon \Xcal\to\Sc$ be a morphism of superschemes.
\begin{enumerate}
\item If $f\colon \Xcal\to\Sc$ is flat (resp. faithfully flat), then the induced morphism $f\colon X\to S$ between the underlying ordinary schemes is also flat (resp. faithfully flat).
\item $f\colon \Xcal\to\Sc$
is faithfully flat if and only if it is flat and for every quasi-coherent sheaf $\Nc$ is on $\Sc$, the condition $f^\ast\Nc=0$ implies that $\Nc=0$.
\end{enumerate}
\end{prop}
\begin{proof} (1) We have only to prove that $f\colon X\to S$ is flat, since the surjectivity of $f$ is topological. By base change   we can assume that $\Sc=S$. Since the question is local on $\Xcal$  we can also assume that $\Oc_{\Xcal}=Gr(\Oc_{\Xcal})$. Then, for every morphism $N'\to N$ of quasi-coherent $\Oc_S$-modules, $N'\otimes_{\Oc_S}\Oc_X\to N\otimes_{\Oc_S}\Oc_X$ is a direct summand of $f^\ast N'\to f^\ast N$. This implies the statement.

(2) By applying Lemma \ref{lem:vanishX} to $\Sc$, we can assume that the base is an ordinary scheme $ S$. Now it is enough to apply again Lemma \ref{lem:vanishX} for $\Xcal$ and the analogous statement for morphisms of ordinary schemes.
\end{proof}

\begin{defin}\label{def:smoothrel}
A morphism $f\colon \Xcal\to\Sc$ of superschemes is smooth of relative dimension $(m,n)$ if it is flat and for every (closed) point $s\in S$ the fibre $\Xcal_s$ of $f$ over $s$ is a smooth scheme of dimension $(m,n)$
 (Definition \ref{def:smooth}).
\end{defin}

By Proposition  \ref{p:flat},
if $f\colon \Xcal\to\Sc$ is a flat morphism of superschemes, then $f\colon X\to S$ is also flat, so that it is \emph{universally open}, and  the image $\Ima f$ is open in $S$.

\begin{defin}\label{def:etale}
A morphism $f\colon \Xcal\to\Sc$ is \'etale if it is flat and smooth of relative dimension $(0,0)$ over the open sub-superscheme $\Ima f$ of $\Sc$.
An \'etale covering is a surjective \'etale morphism.
\end{defin}

By Proposition \ref{p:flat}, if $f\colon \Xcal\to\Sc$ is smooth or \'etale, the induced morphism $f\colon X\to S$ between the underlying ordinary schemes is smooth or \'etale as well.

\begin{defin}[(\cite{DoHeSa97}, Definition 7)]\label{def:etaletop}
The \emph{\'etale topology} is the Grothendieck topology on the category $\Sf$ of superschemes whose coverings are the surjective \'etale morphisms.
\end{defin}

This allows to generalize to supergeometry some standard constructions and definitions  about the \'etale topology of schemes. In particular, we have the following definition (see \cite{DoHeSa97}).
\begin{defin}\label{def:etalerel} 
An (Artin) algebraic superspace is a sheaf $\Xf$ for the \'etale topology of superschemes that can be expressed as the categorical quotient of an \'etale equivalence relation of superschemes
$$
\Tc \rra \U \to \Xf\,.
$$
(here we are confusing superschemes and their functors of  points).
\end{defin}

\begin{defin} A property P of  morphisms of superschemes is local on the target  for the \'etale topology if a morphism $f\colon \Xcal \to \Sc$ has the property P if and only if for every \'etale covering $\phi\colon \Tc\to \Sc$, the fibre product $\phi^\ast f\colon \Xcal\times_\Sc \Tc \to \Tc$ has the property P.

A property P of morphisms of superschemes is local on the source  for the \'etale topology if a morphism $f\colon \Xcal \to \Sc$ has the property P  if and only if for every \'etale covering $\phi\colon \Tc\to \Xcal$, the composition $f\circ\phi\colon \Tc \to \Sc$ has the property P. 
\end{defin}

\begin{prop}\label{p:etalelocal}
The properties of being flat, smooth and \'etale, are local on the target and on the source for the \'etale topology of superschemes. The properties of being separated and proper are local on the target for the same topology.
\end{prop}
\begin{proof}
The properties of being flat, separated, proper,  smooth and \'etale are stable  under  base change.  The remaining   properties follow from the fact that \'etale covers are faithfully flat. 
\end{proof}

\subsection{Deformations of superschemes}

Let $\Xcal_0$ be a superscheme and let $\Sc$ be a superscheme with a distinguished closed point $s_0\in S$. 
\begin{defin} A deformation of $\Xcal_0$ with base the pointed superscheme $(\Sc,s_0)$ is a relative superscheme $\pi\colon \Xcal \to \Sc$ with an isomorphism between the fibre of $\pi$  over $s_0$ and $\Xcal_0$, that is, such that there is a fibred diagram
$$
\xymatrix{ \Xcal_0 \ar@{^{(}->}[r]\ar[d] & \Xcal\ar[d]^{\pi} \\
\Spec k \ar@{^{(}->}[r]^-{s_0}& \Sc
}
$$
Two deformations $\pi\colon \Xcal \to \Sc$, $\pi'\colon \Xcal' \to \Sc$ are \emph{equivalent} if there is an isomorphism $\varpi\colon \Xcal'\to\Xcal$ of $\Sc$-superschemes inducing the identity on the fibres over $s_0$.

A deformation is said to be \emph{local} if $S$ is the spectrum of a local ring, and \emph{infinitesimal} if $\Sc$ is the superscheme $\Spec k[\epsilon_0,\epsilon_1]:=(\Spec k, k[\epsilon_0,\epsilon_1])$ where $\epsilon_i$ has parity $i$ and
$\epsilon_0^2=\epsilon_1^2=\epsilon_0\epsilon_1=0$.
\end{defin}

If $\gamma\colon \Sc'\to\Sc$ is another superscheme and $s'_0$ is a point of $\Sc'$ such that $\gamma(s'_0)=s_0$, for every deformation $\Xcal\to \Sc$ with base $(\Sc,s_0)$ the fiber product $\Xcal\times_{\Sc}\Sc'\to\Sc'$ is deformation of $\Xcal_0$ with base $(\Sc',s'_0)$. This defines a functor of deformations:
$$
\begin{aligned}
\left\{\text{Pointed superschemes}
\right\} & \to \left\{\text{Deformations of $\Xcal_0$}
\right\}
\\
(\Sc,s_0) &\mapsto 
\mbox{\it Def}_{(\Sc,s_0)}(\Xcal_0)=\left\{\text{Deformations of $\Xcal_0$ with base $(\Sc,s_0)$}\right\}
\end{aligned}
$$

The \emph{trivial} deformation of $\Xcal_0$ with base $\Sc$ is the product deformation
$$
\xymatrix{ \Xcal_0 \ar@{^{(}->}[r]\ar[d] & \Xcal_0\times\Sc\ar[d]^{\pi} \\
\Spec k \ar@{^{(}->}[r]^-{s_0}& \Sc
}
$$
A deformation is \emph{locally trivial} if $X$ can be covered by open subsets $U$ such that 
$$
\xymatrix{ \rest{{\Xcal_0}}{U} \ar@{^{(}->}[r]\ar[d] & \rest{{\Xcal_0}}{U}\ar[d]^{\pi} \\
\Spec k \ar@{^{(}->}[r]^-{s_0}& \Sc
}
$$
is the trivial deformation.
Following the proof of Theorem 1.2.4 in \cite{Ser06}, one has
\begin{prop} \label{p:defsmooth}
Every deformation of a smooth superscheme is locally trivial.
\qed\end{prop}

One easily sees that every locally trivial infinitesimal deformation 
$$
\xymatrix{ \Xcal_0 \ar@{^{(}->}[r]\ar[d] & \Xcal_{\epsilon_0,\epsilon_1}\ar[d]^{\pi} \\
\Spec k \ar@{^{(}->}[r]^-{s_0}& \Spec k[\epsilon_0,\epsilon_1]
}
$$
can be obtained by glueing trivial deformations $\rest{{\Xcal_0}}{U_i}\times \Spec k[\epsilon_0,\epsilon_1]\to \Spec k[\epsilon_0,\epsilon_1]$ on affine open sub-superschemes $U_i$ of $\Xcal_0$ by means of automorphisms
$$
\xymatrix@R=10pt@C=1pt{\rest{{\Xcal_0}}{U_{ij}}\times \Spec k[\epsilon_0,\epsilon_1]\ar[rr]^{\simeq}\ar[rd] & & \rest{{\Xcal_0}}{U_{ij}}\times \Spec k[\epsilon_0,\epsilon_1] \ar[ld]\\ 
&\Spec k[\epsilon_0,\epsilon_1]&
}
$$
of superschemes over $\Spec k[\epsilon_0,\epsilon_1]$ that induce the identity over the closed fibre. But such automorphisms are determined by automorphisms $\alpha$ of $\Oc_{\Xcal_0}(U_{ij})\otimes k[\epsilon_0,\epsilon_1]$ as graded $k[\epsilon_0,\epsilon_1]$-algebras that are the identity modulo $(\epsilon_0,\epsilon_1)$. Then, for $a\in \Oc_{\Xcal_0}(U_{ij})$ and $\lambda\in k[\epsilon_0,\epsilon_1]$, one can write $\alpha(a\otimes\lambda)= \alpha(a\otimes 1)\lambda= (a+ D_{\alpha 0}(a)\epsilon_0+(-1)^{\vert a\vert}D_{\alpha 1}(a)\epsilon_1)\lambda$ and $D_\alpha=D_{\alpha 0}+D_{\alpha 1}$ is a derivation $D_\alpha$ of $\Oc_{\Xcal_0}(U_{ij})$, that is, a section of $\Theta_{\Xcal_0}$ on $U_{ij}$. 
One deduces:
\begin{prop}\label{ksmap1} There is a bijection
$$
\left\{\begin{matrix}\text{Isomorphism classes of locally trivial}\\
\text{infinitesimal deformations of $\Xcal_0$}
\end{matrix}
\right\}\simeq H^1(X,\Theta_{\Xcal_0})\,,
$$
which maps the trivial deformation to zero.
This map is called the \emph{Kodaira-Spencer map} for $\Xcal_0$.
If $\Xcal_0$ is smooth, the Kodaira-Spencer map is a bijection
$$
\mbox{\it Def}_{\mbox{\it \tiny inf}}\,(\Xcal_0)\simeq H^1(X,\Theta_{\Xcal_0})\,.
$$
between the set of the isomorphism classes of all the infinitesimal deformations and the first cohomology group of the tangent sheaf. 
\end{prop}
\begin{proof} See Theorem 1.2.9. of \cite{Ser06} or Theorem 2.1.2 of \cite{Va88}.
\end{proof}

It follows that $\mbox{\it Def}_{\mbox{\it \tiny inf}}\,(\Xcal_0)$ has the structure of a graded $k$-vector space. We shall often refer to $H^1(X,\Theta_{\Xcal_0})$ as to the \emph{space of (automorphism classes of) infinitesimal deformations of} $\Xcal_0$.

One can develop a similar theory for morphisms $\pi_0\colon \Xcal_0\to\Sc_0$ of superschemes, that is, for superschemes over ``general points'' $\Sc_0$. For any superscheme $\Sc$ and any ``point'' $s_0\colon \Sc_0\to \Sc$, a deformation of $\pi_0\colon \Xcal_0\to\Sc_0$ with base $s_0$ is now a morphism $\pi\colon \Xcal\to\Sc$ such that there is a fibred diagram
$$
\xymatrix{ \Xcal_0 \ar[r]\ar[d]^{\pi_0} & \Xcal\ar[d]^{\pi} \\
\Sc_0 \ar[r]^-{s_0}& \Sc
}
$$
One has the natural definitions of trivial, locally trivial and infinitesimal deformations.
Procceeding as above, one see that  Propositions \ref{p:defsmooth} and \ref{ksmap1} remain valid when $\Sc_0$ is affine. 
As a consequence: \begin{prop} \label{relksmap1} Let $\pi_0\colon \Xcal_0\to\Sc_0$  smooth.
If $\mbox{\it Def}_{\mbox{\it \tiny inf}}\,(\Xcal_0/\Sc_0)$ is the sheaf (of pointed sets) of the infinitesimal deformations of $\pi_0\colon\Xcal_0\to\Sc_0$, there is an isomorphism of sheaves of pointed sets
$$
\mbox{\it Def}_{\mbox{\it \tiny inf}}\,(\Xcal_0/\Sc_0)\simeq R^1\pi_{0\ast}(\Theta_{\Xcal_0/\Sc_0})\,,
$$
such that the trivial deformations go to zero. This is called the \emph{relative Kodaira-Spencer map} for $\pi_0\colon \Xcal_0\to\Sc_0$.
\end{prop}

\subsection{The Kodaira-Spencer map for general deformations}

Given a general pointed superscheme $s_0\colon \Sc_0\to\Sc$, we have seen (Proposition \ref{p:tangentdefor}) that there is an identification of sheaves of pointed sets 
$$
\begin{aligned}
s_0^\ast\Theta_{\Sc}&\iso \mathcal{H}om_{s_0}(\Sc_0[\epsilon_0,\epsilon_1], \Sc)
\\
v&\mapsto \phi(v)\,,
\end{aligned}
$$
where the second member stands for the extensions $\Sc_0[\epsilon_0,\epsilon_1]\to\Sc$ of  $s_0$.
If $\pi\colon \Xcal \to \Sc$ is a deformation of a smooth morphism $\pi_0\colon\Xcal_0\to\Sc_0$ with base $s_0$, the pull-back $\phi(v)^\ast (\pi)$ is an infinitesimal deformation of $\pi_0\colon\Xcal_0\to\Sc_0$. According to Proposition \ref{relksmap1} one has:

\begin{prop}\label{ksmap2}  For every deformation $\pi\colon \Xcal \to \Sc$ of $\pi_0\colon\Xcal_0\to\Sc_0$ with base $s_0$ there is a morphism of sheaves of $\Oc_{\Sc_0}$-modules
$$
\begin{aligned}
s_0^\ast\Theta_{\Sc}& \xrightarrow{ks_{s_0}(\pi)} R^1\pi_{0\ast}(\Theta_{\Xcal_0/\Sc_0})\\
v& \mapsto \phi(v)^\ast(\pi)
\end{aligned}
$$
This   map is called the \emph{Kodaira-Spencer map of the deformation} $\pi$.
\qed
\end{prop}

We can define a Kodaira-Spencer map for any smooth morphism  $\pi\colon \Xcal \to \Sc$ of superschemes. By the smothness of $\pi$  there is an exact sequence
$$
0\to \Theta_{\Xcal/\Sc}\to\Theta_{\Xcal}\xrightarrow{d\pi}\pi^\ast\Theta_\Sc\to 0\,.
$$
\begin{defin} \label{ksmap3} The (relative) Kodaira-Spencer map of  $\pi\colon \Xcal\to \Sc$ is the composition 
\begin{equation*}
ks(\pi)\colon \Theta_\Sc \to R^1\pi_\ast(\Theta_{\Xcal/\Sc})
\end{equation*}
of the natural morphism $\Theta_\Sc\to \pi_\ast(\pi^\ast\Theta_\Sc)\iso \Theta_\Sc\otimes \pi_\ast(\Oc_{\Xcal})$ with the connecting morphism $\pi_\ast(\pi^\ast\Theta_\Sc)\to R^1\pi_\ast(\Theta_{\Xcal/\Sc})$. 
\end{defin}

Note that, contrary to what happens for a morphism of ordinary schemes with connected fibres,  the map  $\Oc_\Sc\to \pi_\ast\Oc_\Xcal$ may fail to be an isomorphism.

Assume now that $\pi\colon \Xcal\to \Sc$ is a deformation of a smooth morphism $\pi_0\colon\Xcal_0\to \Sc_0$ with base $s_0\colon \Sc_0\to\Sc$. 
Then the relative Kodaira-Spencer map of  $\pi\colon \Xcal\to \Sc$ determines the Kodaira-Spencer map of $\pi$ as a deformation of $\pi_0$.
Actually one can check that the following diagram   commutes
$$
\xymatrix@C=60pt{
s_0^\ast\Theta_\Sc \ar[r]^(.4){s_0^\ast(ks(\pi))} \ar[rd]_{ks_{s_0}(\pi)} & s_0^\ast R^1\pi_\ast(\Theta_{\Xcal/\Sc}) \ar[d]\\
&R^1\pi_{0\ast}(\Theta_{\Xcal_0/\Sc_0})
}
$$
where the vertical map is the cohomology base change morphism.
We shall refer to $R^1\pi_\ast(\Theta_{\Xcal/\Sc})$ as   the \emph{sheaf of (automorphism classes of) infinitesimal deformations of $\pi\colon \Xcal \to \Sc$}.

\subsection{Supercurves and relative supercurves}\label{supercurves}
For brevity of exposition we adopt a restricted definition of supercurve.

\begin{defin}
A supercurve is a smooth proper superscheme of dimension $(1,1)$.
\end{defin}

By Proposition \ref{prop:curvesplit} a supercurve is always split, and one has 
$$
\Oc_\Xcal\simeq \Oc_X\oplus\Lcl^\Pi\iso \textstyle{\bigwedge_{\Oc_X}}(\Lcl)\,.
$$
for a line bundle $\Lcl$ on $X$. Moreover,  by Equation \eqref{eq:bostangent} one has
\begin{equation} \label{eq:der}
\Theta_{\Xcal\vert X}\iso \Theta_X\oplus (\Lcl^{-1})^\Pi
\end{equation} 
(decomposition into even and odd parts).

For future use, we state here a simple Lemma. 
Let $\Dcal\hookrightarrow \Theta_\Xcal$ be a locally free submodule of rank $(0,1)$. The composition of $\rest{\Dcal}{X} \to \Theta_{\Xcal\vert X}$ and $\Theta_{\Xcal\vert X}\to (\Lcl^{-1})^\Pi$ yields an even morphism $\rest{\Dcal}{X}   \to (\Lcl^{-1})^\Pi$.
\begin{prop} \label{p:locfree}
If the quotient $\Theta_\Xcal/\Dcal$ is locally free,   the above morphism of $\Oc_X$-modules is an isomorphism:
$$
\rest{\Dcal }X\iso (\Lcl^{-1})^\Pi\,.
$$
Moreover, if $(z,\theta)$ are graded local coordinates, there is a local basis $D$ of $\Dcal $ such that $\rest DX\mapsto \rest DX(\theta)\partial$ under this isomorphism, $\partial$ being the dual basis of $\theta$.
\end{prop}
\begin{proof}
Since $\Der(\Oc_\Xcal)/\Dcal$ is locally free, one can locally find derivations $D_0$ and $D_1$, even and odd, respectively, such that $\Dcal=\left<D_1\right>$ and $(D_0,D_1)$ is a graded basis of $\Theta_\Xcal$. Then the restrictions $D_{0\vert X}$, $D_{1\vert X}$ form a graded basis of  $ \Theta_{\Xcal\vert X}$, so that $\Theta_X\simeq \langle D_{0\vert X}\rangle$ and $(\Lcl^{-1})^\Pi\simeq   \langle D_{1\vert X}\rangle=\rest{\Dcal}{X}$ according to the decomposition \eqref{eq:der}, and $\rest DX\mapsto \rest DX (\theta)\partial$ where $\partial$ is the dual basis of $\theta$.
\end{proof} 

\begin{defin}\label{gm}
We say that a supercurve $\Xcal=(X,\Oc_{\Xcal}\simeq \Oc_X\oplus\Lcl^\Pi)$ has genus $g$ and degree $m$ if $g$ is the genus of $X$ and  $m$ is the degree of the line bundle $\Lcl$.
\end{defin}

\begin{defin} \label{def:supercurverel}A supercurve over a superscheme $\Sc$ is a smooth proper morphism $\Xcal\to \Sc$ of superschemes of relative dimension $(1,1)$.
\end{defin}

An easy consequence of Proposition \ref{p:etalelocal} is:
\begin{prop} \label{p:etalesupercurve}
For any morphism $\Xcal\to \Sc$ of superschemes, the property of being a relative supercurve is local on the target for the \'etale topology of superschemes.
\qed\end{prop}

The fibre  over a closed point $s\in S$ is a supercurve $\Xcal_s=(X_s, \Oc_{X_s}\oplus \Lcl_s)$ for a line bundle $\Lcl_s$ on $X_s$. Since $f$ is flat, the genus of $X_s$ and the degree of $\Lcl_s$ (Definition \ref{gm}) are independent of the point $s$ as far as $S$ is connected, so we can define the \emph{genus} $g$ and the \emph{degree} $m$ \emph{of the relative supercurve} $\Xcal\to \Sc$ over a connected base.

Let $\pi\colon \Xcal \to \Sc$ be a (relative) supercurve.
 Consider the diagram
\begin{equation}\label{eq:diagram}
\xymatrix{
X \ar@{^{(}->}[r]
^j\ar[rd]_{\tilde\pi}& \Xcal_S \ar@{^{(}->}[r] ^i\ar[d]^{\pi_S} & \Xcal\ar[d]^\pi\\
& S  \ar@{^{(}->}[r] ^i&\Sc
}
\end{equation}

Proceeding as in the absolute case one proves the following result:

\begin{prop}\label{p:split}  Every supercurve over an ordinary scheme is split, i.e., $\Oc_{\Xcal_S}\iso\Oc_X\oplus\Lcl^\Pi$ for a line bundle $\Lcl$ on $X$ (here $S$ is an ordinary scheme).
\qed
\end{prop}

The splitting gives a retraction $p\colon \Xcal_S\to X$ of the immersion $j$.
Assume that $\Sc$ is split, $\Sc=(S,\bigwedge\Ec)$, so that the natural closed immersion $i\colon S\hookrightarrow\Sc$ has a retraction $\rho\colon \Sc\to S$, $\rho\circ i=Id$. In this case we have another relative split curve
$$
\rho^\ast \Xcal_S=(X,\bigwedge(\tilde\pi^\ast\Ec\oplus\Lcl^\pi))\to \Sc
$$

\begin{prop}\label{p:locsplit} 
The relative supercurves $\pi\colon\Xcal \to \Sc$ and $\rho^\ast \Xcal_S\to \Sc$ are locally isomorphic on the source, that is, there is a covering of $\Xcal$ by open sub-superschemes $\U\hookrightarrow \Xcal$, such that $\U\iso \rho^\ast \U_S$. That is,
$$
\Oc_{\U}\iso \bigwedge(\rest{\tilde\pi^\ast\Ec}U\oplus \rest{\Lcl}U)\iso \Oc_{\U_S}\oplus \Lcl^\pi_{\U_S}\,,
$$
where $\Lcl_{\U_S}$ is the line bundle on $\U_S$ induced by projection $p\colon\U_S \to U$.
\end{prop}
\begin{proof}
Let $\Jc_\Sc$, $\Jc_\Xcal$ be the ideals generated by the elements of degree 1 in $\Oc_\Sc$ and $\Oc_\Xcal$, so that $\Ec=\Jc_\Sc/\Jc_\Sc^2$ and $\Oc_\Xcal$ is locally isomorphic to $\bigwedge \F$ with $\F=\Jc_\Xcal/\Jc_\Xcal^2$. Since $\Oc_{\Xcal_S}\iso \Oc_X\oplus \Lcl^\Pi$ by Proposition \ref{p:split}, $\Lcl^\Pi$ is the ideal of $X$ into $\Xcal_S$, and one has
$$
0\to \Jc_\Sc\cdot\Oc_\Xcal\to \Jc_\Xcal \to \Lcl\to 0\,.
$$
as sheaves on $X$. The condition $\Lcl^2=0$ in $\Oc_{\Xcal_S}$ gives $\Jc_\Xcal^2\subseteq \Jc_\Sc\cdot\Oc_\Xcal$, and then, $\Jc_\Sc^2\cdot\Oc_\Xcal=\Jc_\Xcal^2$. It follows that we have an exact sequence of $\Oc_X$-modules
$$
0 \to \Jc_\Sc\cdot\Oc_\Xcal/\Jc_\Sc^2\cdot\Oc_\Xcal \to \F \to \Lcl\to 0\,.
$$
Moreover, by flatness $\Jc_\Sc\cdot\Oc_\Xcal/\Jc_\Sc^2\cdot\Oc_\Xcal\iso \tilde\pi^\ast \Ec$. Since $\Lcl$ is locally free, the exact sequence is locally split, so that there is an open subset $U\subseteq X$ such that 
$$
\rest\F{U}\iso \rest {\tilde\pi^\ast\Ec}{U}\oplus \rest{\Lcl}{U}.
$$
Taking $U$ small enough, we can assume that $\Oc_\U\iso \bigwedge \rest\F{U}$ and the result follows.
\end{proof}

\begin{corol}[Local relative graded coordinates]\label{c:relcoor} Let $\pi\colon \Xcal \to \Sc$ be a (relative) supercurve. Take a covering of $\Sc$ by splitting neighbourhoods $\Vc$. Locally on $\pi^{-1}(\Vc)$, $\pi$ is described by relative graded coordinates $(z,\theta)$, where $z\in \Oc_X(U)$, and $\theta$ is a local basis of $\Lcl^\Pi_{\Vc\vert U}$.
\qed\end{corol}

\bigskip
\section{Relative SUSY curves with \RR\ punctures}\label{s:RRsusy}

Our next step is to define supersymmetric curves with \RR\ punctures. SUSY curves are
supercurves with a conformal structure, i.e., an odd distribution which is maximally nonintegrable.
SUSY curves with \RR\ punctures   have  a ``quasi'' conformal structure: an odd distribution which is maximally nonintegrable  in the complementary of a superdivisor, which is interpreted as the \RR\ puncture.
Thus, in spite of their name, SUSY curves with \RR\ punctures are not SUSY curves, but rather degenerations of them.

\subsection{SUSY curves}
We recall the definition of a SUSY curve.
\begin{defin} \label{def:SUSYrel}
We say that a relative supercurve $\pi\colon\Xcal\to \Sc$ is  \emph{supersymmetric, or a SUSY curve}, if is there is a locally free submodule $\Dcal\hookrightarrow \Theta_{\Xcal/\Sc}$ of rank $(0,1)$ such that 
the composition map
$\Dcal\otimes_{\Oc_{\Sc}}\Dcal \xrightarrow{[\,,\,]}  \Theta_{\Xcal/\Sc}\to  \Theta_{\Xcal/\Sc}/\Dcal$
induces an isomorphism of $\Oc_{\Xcal}$-modules
$$
\Dcal\otimes_{\Oc_{\Xcal}}\Dcal\iso  \Theta_{\Xcal/\Sc}/\Dcal\,.
$$
$\Dcal$ is a \emph{conformal structure} for $\pi\colon\Xcal\to \Sc$.
\end{defin}

\subsection{Superdivisors}

To define \RR\ punctures we need a suitable definition of superdivisor.
Let $\Sc=(S,\Oc_\Sc)$ be a superscheme and $\pi\colon \Xcal\to \Sc$ a relative supercurve (Definition \ref{def:supercurverel}).

\begin{defin}
A relative positive superdivisor of degree $n$ is a closed sub-superscheme $\Zc=(Z,\Oc_\Zc)\hookrightarrow \Xcal$ whose ideal is a line bundle $\Oc_\Xcal(-\Zc)$ of rank $(1,0)$ and whose structure sheaf $\Oc_\Zc$ is a finite flat $\Oc_\Sc$-module of rank $(n,n)$ (see \cite{DoHeSa93}, Definition 5).
\end{defin}

If $(z, \theta)$ is a system of relative graded local coordinates  (Corollary \ref{c:relcoor}), $\Oc_\Xcal(-\Zc)$ is locally generated by an element of type 
\begin{equation}\label{eq:localeq}
f=z^n+(a_1+b_1\theta)z^{n-1}+\dots+(a_n+b_n\theta)\,,
\end{equation}
 where the $a_i's$ are even and the 
$b_i's$ are odd elements in $\Oc_\Sc$. Here $n$ is the degree of the restriction of $\Zc$ to the open sub-superscheme of $\Xcal$ where $(z,\theta)$ are graded coordinates.

Since the projection $\Zc\to \Sc$ is finite and flat, the underlying scheme morphism $Z\to S$ is a covering of degree $n$. Then, there is a dense open subset $S'$ of $S$ such that the fibre $Z_s$ over $s\in S'$ is a reduced divisor of $X_s$, i.e., it  consists of $n$ different points, $Z_s=x_1+\dots +x_n$. The fibre over a point $s$ of the ramification locus $S-S'$ (which is closed of codimension 1 by Zariski's purity theorem \cite{Zariski-Purity}) is a divisor $Z_s=p_1 x_1+\dots +p_q x_q$ with $n=p_1 +\dots +p_q$ and at least one of the coefficients greater than 1.
If we restrict $\Xcal \to \Sc$ and $\Zc$ to the nonramification locus $\Sc'=\rest{\Sc}{S'}$
we have a nonramified superdivisor $\Zc'$ of $\Xcal'\to \Sc'$.  Then, we can find a system of relative graded local coordinates around each point in the support of the superdivisor  so that the equation of the divisor is
\begin{equation}\label{eq:localeq1}
z=0\,.
\end{equation}
To see this, one notes that for suitable $(z, \theta)$ Equation \eqref{eq:localeq} takes the form
$f=z + b\theta$
for an odd element $b$ in $\Oc_{\Sc'}$. Since $(z+b\theta,\theta)$ is also a  system of relative graded local coordinates, we get the desired expression.

This means that the splitting $\Oc_{\Xcal_S}\iso\Oc_X\oplus\Lcl$ and the corresponding retraction  $p\colon \Xcal_S\to X$  ca be chosen so that the local isomorphism $\Xcal\iso \rho^\ast \Xcal_S$ given by Proposition \ref{p:locsplit} transforms $\Zc$ to the superdivisor $\bar\rho^\ast Z$ (where $\bar\rho=p\circ \rho$) induced by the underlying ordinary divisor $Z\hookrightarrow X$.
In other words,

\begin{prop} \label{p:localsplitdiv} Let $\Zc$ a relative superdivisor on a supercurve $\pi\colon \Xcal\to \Sc$ such that $\Zc\to \Sc$ is nonramified.  Then the pair $(\Xcal \to \Sc, \Zc)$ is locally (on the source $\Xcal$) isomorphic to $(\rho^\ast \Xcal_S\to\Sc, \bar\rho^\ast Z)$, that is, locally  one has
$$
\Oc_{\Xcal}\iso \bigwedge(\tilde\pi^\ast\Ec\oplus \Lcl)\iso \Oc_{\Xcal_S}\oplus \Lcl_{\Xcal_S}\,,
$$
where $\tilde\pi\colon X \to S$ is the underlying family of ordinary curves, $\Lcl_{\Xcal_S}=p^\ast\Lcl$ and  $$\Oc_\Xcal(-\Zc)\iso \bar\rho^\ast \Oc_X(-Z).$$ One can choose relative graded coordinates $(z,\theta)$ such that $\Zc$ is given by the equation $z=0$.
\qed
\end{prop}

\subsection{\RR\ SUSY curves}

\begin{defin} \label{def:RRrel}
Let $\pi\colon\Xcal\to \Sc$ be a relative supercurve, and $\Zc\hookrightarrow \Xcal$ a positive relative superdivisor. We say that $\pi\colon\Xcal\to \Sc$ is a \emph{\RR-SUSY curve (or an RR-SUSY curve) along $\Zc$}  if   there is a locally free submodule $\Dcal\hookrightarrow \Theta_{\Xcal/\Sc}$ of rank $(0,1)$ such that 
the composition map
$\Dcal\otimes_{\Oc_{\Sc}}\Dcal \xrightarrow{[\,,\,]}  \Theta_{\Xcal/\Sc}\to  \Theta_{\Xcal/\Sc}/\Dcal$
induces an isomorphism of $\Oc_{\Xcal}$-modules
$$
\Dcal\otimes_{\Oc_{\Xcal}}\Dcal\iso  (\Theta_{\Xcal/\Sc}/\Dcal)(-\Zc)\,.
$$
We then say that $\Dcal$ is a
 \emph{\RR\ conformal structure} for $(\pi\colon\Xcal\to\Sc,\Zc)$ or a  \emph{\RR\ conformal structure for $\pi\colon\Xcal\to\Sc$ along $\Zc$}.
\end{defin}

The \emph{irreducible components} of the superdivisor $\Zc$ are usually called the \emph{\RR\ punctures}.
Note that even these components may be ramified coverings of the base, that is, may intersect a fibre in a divisor with some points counted more than once.

By Proposition \ref{p:localsplitdiv}, if  $\Zc$ is \emph{nonramified} over $\Sc$ locally on $\Xcal$ one has
$$
\Oc_{\Xcal}\iso \bigwedge(\tilde\pi^\ast\Ec\oplus \Lcl)\iso \Oc_{\Xcal_S}\oplus \Lcl^\Pi_{\Xcal_S}\,,
$$
 and one can choose relative graded coordinates $(z,\theta)$ such that $\Zc$ is given by the equation $z=0$.

Let $\pi\colon\Xcal\to \Sc$ be a relative supercurve and $\Zc\hookrightarrow \Xcal$ a positive relative superdivisor.
If $\Dcal\hookrightarrow \Theta_{\Xcal/\Sc}$ is a locally free submodule of rank $(0,1)$ and 
$$
\Omega_{\Xcal/\Sc} \xrightarrow{\tau} \Bc:=\Dcal^{-1}\to 0
$$
is the dual projection,   the quotient sheaf $\Theta_{\Xcal/\Sc}/\Dcal$ is locally free of rank $(1,0)$ if and only if $\Dcal=\ker \tau$. Conversely, given a locally free quotient $\tau\colon \Omega_{\Xcal/\Sc} \to \Bc$ of rank $(0,1)$, the sheaf $\Dcal=\Ima \tau^\ast$ is a submodule of rank $(0,1)$ of $\Theta_{\Xcal/\Sc}$ and $\Theta_{\Xcal/\Sc}/\Dcal$ is locally free. Moreover, $\ker\tau$ is the submodule of the 1-forms vanishing on $\Dcal$.

\begin{prop}\label{p:caractberRR} 
 $\Dcal$ is a \RR\ conformal structure along $\Zc$ if and only if the composition
$$
\ker\tau \hookrightarrow  \Omega_{\Xcal/\Sc}\xrightarrow{d}  \Omega_{\Xcal/\Sc}\wedge \Omega_{\Xcal/\Sc} \xrightarrow{\tau\wedge\tau} \Bc^{\otimes 2}
$$
yields an isomorphism $\ker\tau\iso \Bc^{\otimes 2}(-\Zc)$. Moreover, in this case there is an isomorphism 
$$
\Bc= \Dcal^{-1}\iso\Ber(\Xcal/\Sc)(\Zc)\,,
$$
where $\Ber(\Xcal/\Sc) =\Ber(\Omega_{\Xcal/\Sc})$ is the relative Berezinian sheaf.
 
%
%
\end{prop}
\begin{proof} 
 Since $\Theta_{\Xcal/\Sc}/\Dcal$ is locally free, we can locally find  a   basis $(D_0,D_1)$ of $\Theta_{\Xcal/\Sc}$ with $D_0$ even and $D_1$ odd such that $\Dcal$ is   generated by $D_1$. If $(\omega_0,\omega_1)$ is the dual basis, $\ker\tau$ is generated by $\omega_0$ and $(\tau\wedge\tau)(d\omega_0)= d\omega_0(D_1,D_1)\tau(\omega_1)^{\otimes 2}= \omega_0([D_1,D_1])\tau(\omega_1)^{\otimes 2}$, as $\omega_0(D_1)=0$. If we write $[D_1,D_1]=fD_0+gD_1$, we get $(\tau\wedge\tau)(d\omega_0)=f \tau(\omega_1)^{\otimes 2}$. Now each of the two conditions of the statement is equivalent to $f$ being of the form $f=h\cdot e_Z$ where $h$ is invertible and $e_Z$ is a local generator of $\Oc_{\Xcal}(-\Zc)$. 
 
Since $\Ber(\Nc)\iso \Nc$ for a line bundle of rank $(1,0)$ and   $\Ber(\Nc)\iso \Nc^{-1}$ for a line bundle of rank $(0,1)$, taking Berezinians in the dual of the exact sequence of locally free sheaves
$$
0\to\Dcal\to \Theta_{\Xcal/\Sc} \to \Dcal\otimes_{\Oc_{\Xcal}}\Dcal(\Zc)\to 0
$$
we  obtain $\Ber(\Xcal/\Sc)\iso \Dcal^{-1}(-\Zc)$, which finishes the proof. 
\end{proof}

Therefore, RR-conformal structures may also be described in terms of the relative Berezinian sheaf (see also \cite{Del87}).

For SUSY curves, one easily sees that the local coordinates can be chosen so that $\Dcal$ is locally generated by $\frac{\partial}{\partial \theta}+\theta \frac{\partial}{\partial z} $. For RR-SUSY curves one only has:

\begin{prop}\label{p:localsplitRR}  Let $\pi\colon\Xcal\to \Sc$ be a relative curve, $\Zc\hookrightarrow \Xcal$ a positive relative superdivisor nonramified over $\Sc$, and $\Dcal$ a \RR\ conformal structure along $\Zc$. There exists an \'etale covering $\Tc\to\Sc$ with the following property: on the base-changed RR-SUSY curve $(\Xcal_{\Tc},\Zc_{\Tc},\Dcal_{\Tc})\to\Tc$ locally (in the Zariski topology of the source) there are  relative graded coordinates $(z,\theta)$ such that
$\Zc_{\Tc}$ is given by the equation $z=0$, and $\Dcal_\Tc$ is locally generated by
\begin{equation}\label{eq:localD}
D=\frac{\partial}{\partial \theta}+z\theta \frac{\partial}{\partial z}\,.
\end{equation}
These coordinates are called \emph{superconformal}. Then, an RR-SUSY curve is locally  trivial on the source \emph{in the \'etale topology}, that is, locally on $\Xcal_{\Tc}$ it is isomorphic to
$\Xcal_{\Tc}=(X,  \Oc_{\Xcal_T}\oplus \theta\Oc_{\Xcal_T})\to\Tc$ equipped with the superdivisor $z=0$ and the distribution 
$$
\Dcal=\langle \frac{\partial}{\partial \theta}+z\theta \frac{\partial}{\partial z} \rangle\,.
$$
\end{prop}

\begin{proof} We can assume that $S$ is affine. Take first $(w,\eta)$ such that, according to Proposition \ref{p:localsplitdiv},
 $\Zc$  has equation  $w=0$. Since $\Theta_{\Xcal/\Sc}/\Dcal$ is locally free, tensoring by an invertible sheaf if necessary, one can take a local generator  of $\Dcal$ of the form
$$
D'=\frac{\partial}{\partial \eta}+g \frac{\partial}{\partial w}\,,
$$
where $g$ is odd. Then $(\Theta_{\Xcal/\Sc}/\Dcal)(-\Zc)$ is generated by the class of $w\frac{\partial}{\partial w}$. Since
$$
\frac 12 [D',D']= \bigg(\frac{\partial g}{\partial \eta}+g \frac{\partial g}{\partial w}\bigg) \frac{\partial }{\partial w}=\frac{\partial g}{\partial \eta} \frac{\partial}{\partial w} \,,
$$
where the second equality follows from $g^2=0$, the condition of $\Dcal$ being a superconformal structure along $\Zc$ is equivalent to 
$$
\frac{\partial g}{\partial \eta}=w\beta\,,
$$
with $\beta$ even and invertible and not depending on $\eta$. Let us write $\beta=\beta_0+w\hat\beta$ where $\beta_0$ is an invertible function in $\Oc_\Sc(S)$.

We can now take an \'etale covering of the base such that there exist $\alpha$ invertible in $\Oc(\Tc)$ such that $\alpha^2=\beta_0$. So we can take $T=\Spec \Oc(S)[t]/(t^2-\beta_0)$ and $\Oc_\Tc=\Oc_\Sc[t]/(t^2-\beta_0)$.

If we consider the local coordinates $(w,\theta=\alpha\eta)$, the  local generator $D=\alpha^{-1}D'$ of $\Dcal$ takes the form
$$
D=\frac{\partial}{\partial \theta}+f \frac{\partial}{\partial w}
$$
where the coefficient $f$  satisfies
$$
\frac{\partial f}{\partial \theta}=w\gamma\,, \quad \text{with $\gamma=1+w\hat\gamma$}\,,
$$
$\hat\gamma$ being an even local section of $\Oc_{\Xcal_S}\iso \bigwedge(\tilde\pi^\ast\Ec)$.

We want to find new coordinates $(z,\theta)$ such that  $D$  
can be written as
$$
D=\frac{\partial}{\partial \theta}+z\theta \frac{\partial}{\partial z}\,.
$$
Since
$$
D=\frac{\partial}{\partial \theta}+\big(\frac{\partial z}{\partial\theta}+f \frac{\partial z}{\partial w}\big)\frac{\partial}{\partial z}
$$
 we  must  solve
$$
z\theta=\frac{\partial z}{\partial\theta}+f \frac{\partial z}{\partial w}
$$
for $z$. 
Let us write $f=f_1+\theta w\gamma$, $z=z_0+\theta z_1$, where the subscripts $0$ and $1$ stand for even and odd local sections of $\Oc_{\Xcal_S}\iso \bigwedge(\tilde\pi^\ast\Ec)$. We have two equations
$$
0=z_1+f_1\frac{\partial z_0}{\partial w}\,,\quad  z_0=w\gamma \frac{\partial z_0}{\partial w} - f_1 \frac{\partial z_1}{\partial w}\,.
$$
From the first equation we obtain $\frac{\partial\bar z_1}{\partial w}=-(\frac{\partial f_1}{\partial w} z_0+f_1\frac{\partial z_0}{\partial w})$ and then $f_1 \frac{\partial z_1}{\partial w}=0$, since $f_1^2=0$ and $f_1\frac{\partial f_1}{\partial w}=0$. The second equation now gives 
$$
 z_0=w\gamma \frac{\partial z_0}{\partial w}
$$
which has a solution as $\gamma=1+w\hat\gamma$.

To conclude, note that although the \'etale covering of the base depends on the coordinate neighbourhood chosen at the beginning of the proof, as  $\pi$ is proper there exists an \'etale covering of the base that works for all    coordinate neighbourhoods.
\end{proof}

\begin{remark}\label{rem:localsplitRR}  The 
disjoint union of the
Zariski open sub-superschemes
$\U$ of $\Xcal_\Tc$ where superconformal coordinates exist is an \'etale covering of   the original $\Xcal$. In the case of a single RR-SUSY curve  no \'etale covering of the base is needed (any connected \'etale covering of $\Spec k$ is the identity).
\end{remark}

\subsection{RR-SUSY curves over ordinary schemes}
 By Proposition \ref{p:split},  any relative supercurve
 $\Xcal \to S$   over an ordinary scheme $S$  is split, so that $\Oc_{\Xcal}\simeq \Oc_X\oplus \Lcl^\Pi$ for a line bundle $\Lcl$ on $X$.
There is a decomposition
\begin{equation}\label{eq:localdif}
\Omega_{\Xcal/S}\iso (\kappa_{X/S}\otimes_{\Oc_X}\Oc_\Xcal)\oplus (\Lcl^\Pi\otimes_{\Oc_X}\Oc_\Xcal)\,,
\end{equation}
of the relative differentials as sum of locally free sheaves of rank $(1,0)$ and $(0,1)$. Here we   write $\kappa_{X/S}=\Omega_{X/S}$ for the relative canonical line bundle of $X\to S$, and the parity-change functor establishes the correct parity. One then has
$$
\Ber(\Xcal/S)\iso (\kappa_{X/S}\otimes\Lcl^{-1})^\Pi\otimes_{\Oc_X}\Oc_\Xcal\,,
$$ 
so that the decomposition of $\Ber(\Xcal/S)$ as an $\Oc_X$-module into even and odd components is
$$
\Ber(\Xcal/S)\iso \kappa_{X/S}\oplus (\kappa_{X/S}\otimes\Lcl^{-1})\,.
$$

If $\Zc\hookrightarrow \Xcal$ is a positive relative divisor (that may be ramified), it follows from Equation \eqref{eq:localeq} that $\Zc$ is determined by the ordinary relative divisor $Z=\rest{\Zc}X$ to $X\to S$; then $\Oc_\Xcal(-\Zc)\iso \Oc_X(-Z)\otimes_{\Oc_X}\Oc_\Xcal$.

\begin{prop}\label{p:spinrel} 
The existence of an RR-conformal structure $\Dcal$ on $\Xcal\to S$ along $\Zc$  is equivalent to the existence of an isomorphism of $\Oc_X$-modules
$$
\Lcl\otimes\Lcl \iso\kappa_{X/S}(Z)\,.
$$
These isomorphisms are called \emph{\RR\ spin structures} or \emph{RR-spin structures}. That is, the existence of a \RR\ conformal structure $\Dcal$ on $\Xcal\to S$ along $\Zc$ is equivalent to the existence of an RR-spin structure on $X\to S$ along $Z$.
\end{prop}
\begin{proof} If $\Dcal$ is an RR-conformal structure, by Proposition \ref{p:locfree}  one has $\rest{\Dcal}X\iso (\Lcl^{-1})^\Pi$, and  $\Ber(\Xcal/\Sc)\iso \Dcal^{-1}(-\Zc)$ by Proposition \ref{p:caractberRR}. Then $\kappa_{X/S}\otimes \Lcl^{-1}\iso \Lcl(-Z)$.

 For the converse, if  there is an isomorphism $\Lcl\otimes\Lcl \iso\kappa_{X/S}(Z)$, one has an induced isomorphism $\Lcl^{-1}\iso \kappa_{X/S}\otimes \Lcl (-Z)\simeq \Der_{\Oc_X}(\Oc_X,\Lcl(-Z))$, which gives a morphism
$$\gamma\colon \Oc_\Xcal \otimes (\Lcl^{-1})^\Pi \to \Theta_\Xcal(-\Zc)\hookrightarrow  \Theta_\Xcal\,.
$$
Moreover, the sections of $\Lcl^{-1}$ can be considered as odd derivations of $\Oc_\Xcal$, and we have then an injective morphism
$$
\begin{aligned}
\Oc_\Xcal \otimes (\Lcl^{-1})^\Pi & \hookrightarrow \Theta_\Xcal \\
\partial& \mapsto \partial+\gamma(\partial)\,.
\end{aligned}
$$
The image $\Dcal$ is a locally free submodule of rank $(0,1)$ and, if $(z,\theta)$ are (relative) local coordinates in an open subset $U$, $\Dcal$ is locally generated by $D=\frac{\partial}{\partial \theta}+\gamma(\frac{\partial}{\partial \theta})=\frac{\partial}{\partial \theta}+ f e_Z\theta 
\frac\partial{\partial z}$, where $f\in\Oc_X(U)$ is invertible and $e_Z\in\Oc_X(U)$ is a local generator of $\Oc_X(-Z)$. Since $[D,D]=2fe_Z\frac{\partial}{\partial z}$ and $f$ is invertible, $\Dcal$ is an RR-conformal structure along $\Zc$.
\end{proof}

Note that the existence of a \RR\ conformal structure forces the degree $\nf_R$ of $\Zc$ to be \emph{even}, because one has
\begin{equation}\label{eq:degRR}
2m=2g-2+\nf_R
\end{equation}
where $m=\deg \Lcl$.

If $(\Xcal, \Zc, \Dcal)$ is an RR-SUSY curve over $S$, so that $\Oc_\Xcal=\Oc_X\oplus\Lcl$, we write
$$
\Lcl=\kappa_{X/S}(Z)^{1/2}\,.
$$

One also has
\begin{equation} \label{eq:Dtwo}
\begin{aligned}
\Dcal\otimes\Dcal&\simeq (\kappa_{X/S}^{-1}\otimes\Oc_X(-Z))\otimes_{\Oc_X}\Oc_\Xcal\simeq \kappa_{X/S}(Z)^{-1}\oplus \kappa_{X/S}(Z)^{-1/2}\\
\Theta_{\Xcal/S}/\Dcal&\simeq \kappa_{X/S}^{-1}\oplus (\kappa_{X/S}(Z)^{-1/2}\otimes\Oc_X(Z))
\,,
\end{aligned}
\end{equation}
as sheaves of $\Oc_S$-modules, where the direct sum is the decomposition into even and odd parts.

By Proposition \ref{p:localsplitRR}  relative superconformal local coordinates  exist after an \'etale covering of the base. A local computation shows that after that covering
$$
\theta\otimes\theta=z^{-1}dz
$$
under the isomorphism $\Lcl\otimes\Lcl\iso \kappa_{X/S}(Z)$.

\subsection{Morphisms of SUSY curves with \RR\ punctures and base change}

A morphism $\phi\colon \Xcal'\to \Xcal$ of supercurves over $\Sc$
induces a morphism of $\Oc_\Xcal$-modules
$$\phi_\ast\colon \Theta_{\Xcal'/\Sc} \to\phi^\ast \Theta_{\Xcal/\Sc}=
\Theta_{\Xcal/\Sc}\otimes_{\Oc_\Xcal}\Oc_{\Xcal'},$$ so that a conformal
structure $\Dcal'$ on
$\Xcal'/\Sc$ defines a subsheaf $\phi_\ast\Dcal'$ of
$\Theta_{\Xcal/\Sc}\otimes_{\Oc_\Xcal}\Oc_{\Xcal'}$.

\begin{defin} Let $(\Xcal/\Sc,\Zc,\Dcal)$ and $(\Xcal'/\Sc,\Zc',\Dcal')$ be
RR-SUSY curves along relative positive superdivisors of degree $m$ over a superscheme $\Sc$. A morphism between them is a morphism $\phi\colon \Xcal'\to \Xcal$ of supercurves over $\Sc$ such
that $\phi_S(\Zc')=\Zc$ and $\phi_\ast\Dcal'\subseteq\phi^\ast\Dcal=\Dcal\otimes_{\Oc_\Xcal}\Oc_{\Xcal'}$.
\end{defin}

We can then consider the category of
RR-SUSY curves of genus $g$ along a relative positive superdivisor of degree $m$ over a superscheme
$\Sc$.
The automorphisms of $(\Xcal/\Sc,\Zc,\Dcal)$ are the automorphisms 
$\tau$ of $\Xcal$ as a supercurve over $\Sc$ such that $\tau_S(\Zc)=\Zc$ and $\Dcal$ and
$\tau_\ast\Dcal$ define the same subsheaf of
$\Theta_{\Xcal/\Sc}$.
In particular, when the base superscheme is an ordinary
scheme $S$, so that $\Oc_\Xcal\simeq\Oc_X\oplus\Lcl$, we have seen that the data $(\Xcal/S,\Zc,\Dcal)$ are equivalent to $(X/S,\Lcl, Z,\varpi)$, where $\varpi\colon \Lcl\otimes\Lcl \iso\kappa_{X/S}\otimes\Oc_X(Z)$ is an isomorphism of $\Oc_X$-modules. Then $\tau$ is just a pair 
$\tau=(\tau_0,\tau_1)$, where $\tau_0$ is an automorphism of the 
underlying family of curves $X/S$ with $\varpi(Z)=Z$, and $\tau_1$ is an automorphism of
$\Lcl$ as a line bundle, such that the isomorphism $\varpi$
 is preserved. In particular, if
$\tau_0=\Id$, then $\tau_1=\pm\Id$.

Let $(\Xcal/\Sc,\Zc,\Dcal)$  be an RR-SUSY curve along a relative positive superdivisor $\Zc$  over a superscheme $\Sc$. If $\phi\colon \Tc\to \Sc$ is a morphism of superschemes, base change gives a relative RR-SUSY curve $\Xcal_\Tc=\phi^\ast\Xcal \to \Tc$ and a relative positive superdivisor $\Zc_\Tc=\phi^\ast\Zc\hookrightarrow \Xcal_\Tc$ over $\Tc$. Moreover, since $\Theta_{\Xcal/\Sc}/\Dcal$ is locally free, $\Dcal_\Tc=\phi^\ast\Dcal$ is a locally free subsheaf of $\Theta_{\Xcal_\Tc/\Tc}$ of rank $(0,1)$ and it is an RR-conformal structure for $\Xcal_\Tc \to \Tc$ along $\Zc_\Tc$. 
We then have a relative RR-SUSY curve $\phi^\ast (\Xcal/\Sc,\Zc,\Dcal)=(\Xcal_\Tc/\Tc, \Zc_\Tc,\Dcal_\Tc)$ \emph{obtained by the base-change $\phi\colon \Tc\to \Sc$}.
In particular, if $i\colon S\hookrightarrow \Sc$ is the natural immersion of the underlying ordinary scheme, we obtain  a relative RR-SUSY curve  $(\Xcal_S/S,\Zc_S,\Dcal_S)$ over $S$, i.e., an RR-SUSY structure on the split supercurve $\Xcal_S\to S$.
By Proposition \ref{p:spinrel}, the restricted RR-SUSY curve $(\Xcal_S/S,\Zc_S,\Dcal_S)$ is equivalent to  a \RR-spin structure, that is, $\Oc_{\Xcal_S}\iso \Oc_X\oplus\Lcl$ for a line bundle $\Lcl$ on $X$, the superdivisor $\Zc_S$ is determined by its restriction $Z$ to $X$, 
$\Oc_{\Xcal_S}(\Zc_S)\iso \Oc_{X}(Z)\otimes_{\Oc_X}\Oc_{\Xcal_S}$,  the restriction $\rest{\Dcal}X$ of $\Dcal$  (and also of $\Dcal_S$) to $X$ is isomorphic to $\Lcl^{-1}$, and one has $\Lcl\otimes\Lcl\iso\kappa_{X/S}(Z)$.

\subsection{Infinitesimal deformations of an RR-SUSY curve}

Let $(\pi\colon\Xcal\to\Sc,\Dcal, \Zc)$ be an RR-SUSY curve.
From now on,  \emph{we always assume that $\Zc\to\Sc$ is nonramified}, so that the underlying divisor $Z$ intersects every fibre in $m$ different points. This ensures that every infinitesimal deformation of the pair $Z\hookrightarrow X\to S$ is locally trivial, something that we shall not use explicitly but will have an effect on the computation of the dimension of the moduli space using infinitesimal deformations. 

\begin{defin}\label{def:infdefRR} We define a sheaf 
$\Gc_\pi^{\nf_R}$ whose sections on an open set $U$ are the graded derivations $D'\in\Gamma(U,\Theta_{\Xcal/\Sc})$ that preserve $\Dcal$, that is,
$$
\Gc_\pi^{\nf_R}(U)=\{ D'\in \Gamma(U,\Theta_{\Xcal/\Sc}) \,\vert\, [D',D]\in \Dcal(U),\text{for every $D\in \Dcal(U)$}\}\,.
$$
When $\Xcal$ is a single curve we shall  use the notation $\Gc^{\nf_R}_\Xcal$.
\end{defin}

If we consider a single RR-SUSY curve $(\Xcal,\Dcal, \Zc)$,  from the discussion  before  Proposition \ref{ksmap1} we see that  the automorphisms of $\rest{{\Xcal_0}}{U_{ij}}\times \Spec k[\epsilon_0,\epsilon_1]$ that preserve the induced distribution $\Dcal\otimes 1$ are the ones defined by  $\alpha(a\otimes\lambda)= (a+ D'_{\alpha 0}(a)\epsilon_0+(-1)^{\vert a\vert}D'_{\alpha 1}(a)\epsilon_1)\lambda$ where  $D'_\alpha=D'_{\alpha 0}+D'_{\alpha 1}$ is a section of $\Gc_\Xcal^{\nf_R}$.
This follows from the expression $\alpha (\bar D)=\alpha\circ \bar D\circ \alpha^{-1}$ for any derivation $\bar D$ of $\Oc_\Xcal(U_{ij})[\epsilon_0,\epsilon_1]$. 

Summing up, one has:

\begin{prop}\label{relksmap1RR} Statements  analogous to Propositions \ref{ksmap1}  and \ref{relksmap1} about infinitesimal deformations and Kodaira-Spencer maps remain true  for RR-SUSY curves if one considers   $\Gc_\pi^{\nf_R}$ instead of the relevant cotangent sheaf.
\qed
\end{prop}

\begin{prop}\label{def:infdef} Let $(\Xcal,\Dcal, \Zc)$ be a single RR-SUSY curve such that $Z$ consists of $\nf_R$ different points.
\begin{enumerate}
\item
There is an isomorphism
$$
\Gc_\Xcal^{\nf_R}\iso (\Theta_{\Xcal}/\Dcal)\otimes_{\Oc_\Xcal}\Oc_\Xcal(-Z)
$$
as sheaves of graded vector spaces. 
\item
There is an isomorphism of sheaves of graded vector spaces
$$
\Gc_\Xcal^{\nf_R}\iso \Dcal\otimes\Dcal\iso \kappa_{X}(Z)^{-1}\oplus \kappa_{X}(Z)^{-1/2}
\,,
$$ 
where we have written $\Lcl=\kappa_{X}(Z)^{1/2}$.
\end{enumerate}
\end{prop}
\begin{proof} Let us denote by $\bar\pi$ the projection $\Theta_{\Xcal/S}\to \Theta_{\Xcal/S}/\Dcal$.

(1) Working locally in superconformal coordinates (Proposition \ref{p:localsplitRR} and Remark \ref{rem:localsplitRR}), let $D'=a\frac{\partial}{\partial z}+bD$ be a graded derivation. Then
\begin{equation}\label{dprimad}
[D',D]= (-a\theta-(-1)^{\vert a\vert }D(a)+2bz)\frac{\partial}{\partial z}+(-1)^{\vert b\vert }D(b) D\,,
\end{equation}
so that $D'$ is a section of $\Gc^{\nf_R}$
if and only if 
\begin{equation}
0=-a\theta-(-1)^{\vert a\vert }D(a)+2bz\,.
\label{eq:bzn}
\end{equation}
If $D'$ is even, then $a=a(z)$ and $D(a)=z\theta a'$, so that Equation \eqref{eq:bzn} gives $0=2bz-a\theta-z\theta a'$. Thus, $a$ is a multiple of $z$, so that $a= z \bar a(z)$, and $\bar\pi$ sends $D'$ to $z\bar a\bar\pi(\frac{\partial}{\partial z})$ and, since $b$ is determined by $a$, $\bar\pi$ induces an isomorphism
$$
\Gc^{\nf_R}_{\Xcal,0}\iso (\Theta_{\Xcal}/\Dcal)_0(-Z)\,.
$$
If $D'$ is odd, then $a=a_1(z)\theta$ and $D(a)=a_1$, which gives $0=2bz+a_1$. Then $a_1$ is a multiple of $z$ and one sees that $\bar\pi$ induces an isomorphism of $\Gc^{\nf_R}_1$ with the multiples of $z$ in  $(\Theta_{\Xcal}/\Dcal)_1$, that is, an isomorphism
$$
\Gc^{\nf_R}_{\Xcal,1}\iso(\Theta_{\Xcal}/\Dcal)_1(-Z)\,.
$$

(2) follows from Equation \eqref{eq:Dtwo}.
\end{proof}

\begin{corol}
$H^0(X,\Gc_\Xcal^{\nf_R})=0$ and
$$
\dim H^1(X,\Gc_\Xcal^{\nf_R})=(3g-3+\nf_R,2g-2+\nf_R/2)\,.
$$
\qed
\end{corol}

We can now compute the infinitesimal deformations of an RR-SUSY curve.

\begin{corol}\label{cor:h1dimrel} Let $(\pi\colon\Xcal\to S,\Dcal, \Zc)$ be an RR-SUSY curve of genus $g$ over an ordinary scheme $S$. One has
$$
\begin{aligned}
\pi_\ast\Gc_\pi^{\nf_R}&=0\\
R^1\pi_\ast \Gc_\pi^{\nf_R} &\simeq [R^1\pi_\ast(\kappa_{X/S}(Z)^{-1})]\oplus [R^1\pi_\ast( \kappa_{X/S}(Z)^{-1/2})] \\
&\simeq  [R^1\pi_\ast(\kappa_{X/S}(Z)^{-1})]\oplus [R^1\pi_\ast( \Lcl^{-1})] \,,
\end{aligned}
$$
 where the direct sum is the decomposition between even and odd parts. Moreover,
$R^1\pi_\ast\Gc_\pi^{\nf_R}$ is a locally free $\Oc_S$-module of rank
$$
\rk (R^1\pi_\ast\Gc_\pi^{\nf_R})=(3g-3+\nf_R, 2g-2+\nf_R/2)\,.
$$
\qed
\end{corol}

Proceeding as in \cite[Proposition 2.2]{LeRoth88}, one has:
\begin{corol}\label{cor:h0dimrel}
 Let $(\pi\colon\Xcal \to \Sc,\Zc,\Dcal)$ be a relative RR-SUSY curve, where now $\Sc$ is any superscheme. One has
$$
\pi_\ast\Gc_\pi^{\nf_R}=0\,.
$$
\qed
\end{corol}

\bigskip
\section{Supermoduli of RR-SUSY curves}\label{s:modRR}

In the section we describe both a local and a global construction of the supermoduli space of \emph{RR-SUSY curves}.  From now on we assume that the base field $k$ is the field of the complex numbers; this will be required in Lemmas \ref{lem:LBR1}  and \ref{lem:LBR2}.
To simplify the exposition we shall only consider curves of genus $g\geq 2$ with a level $n$ structure with $n\ge 3$, so that they have no automorphisms but the identity and there exists a fine moduli scheme $M_g$ of ordinary curves, carrying  a universal curve $X_g\to M_g$.\footnote{A level $n$ structure on an ordinary curve $p\colon X \to S$ 
is an isomorphism between the
$n$-torsion of the relative Jacobian of the curve and the group $\Gamma(S,R^1p_\ast \Z_n)$.}
We also fix a positive integer $\nf_R$ which is the degree of the divisor of the relative positive superdivisor along which the \RR\ punctures are supported. The supermoduli of RR-SUSY curves is constructed in the same way as the supermoduli for SUSY curves \cite{DoHeSa97}. The steps of the constructions are the following:

\begin{enumerate}
\item
Construction of the bosonic supermoduli. For that, we mean  the underlying scheme (or the relevant algebraic structure) to the supermoduli, assuming that the latter exists.
\item
Computation of the infinitesimal deformations of RR-SUSY curves.
\item Construction of the ``local supermoduli superscheme'' out of the infinitesimal deformations.
\item Extension to RR-SUSY curves  of the results of Le Brun-Rothstein \cite{LeRoth88}, which they proved in the case of SUSY curves.
\end{enumerate}

From now on, we always assume that the positive relative divisors $\Zc$ are \emph{nonramified over the base}.

\begin{defin}\label{def:modfunctorRR} The functor of RR-SUSY curves along a relative positive superdivisor of degree $\nf_R$ is the \emph{sheaf $\F_{\nf_R}$ for the \'etale topology of superschemes} (\cite{DoHeSa97}, Definition {4.1})
associated to the presheaf 
\begin{equation}
\Sc\to F_{\nf_R}(\Sc)=
\left\{
\begin{matrix}
\text{Isomorphism classes of SUSY curves over $\Sc$ with}\\
\text{a \RR\ puncture along a positive}
\\
\text{superdivisor $\Zc$ of relative degree $\nf_R$}\\
\text{such that $\Zc\to\Sc$ is nonramified}
\end{matrix}
\right\}\,.
\end{equation} 
\end{defin}

\subsection{The bosonic Moduli functor of RR-SUSY curves}

In this subsection we describe the \emph{functor of RR-SUSY curves over ordinary schemes}. Due to Proposition \ref{p:spinrel}, RR-SUSY curves over ordinary schemes are related to RR-spin structures. 

\begin{defin}\label{def:RR} An RR-spin curve over an ordinary scheme $S$ 
is a curve $X\to S$ with a relative positive divisor $Z$ nonramified over $S$, and an
``invertible sheaf class''
$\Upsilon\in\Pic X/S$ such that $\Upsilon^2=[\kappa_{X/S}(Z)]$.
\end{defin}

So a curve with an RR-spin structure in the sense of Proposition  \ref{p:spinrel} is an RR-spin curve, but the converse is only true after an \'etale base change, that is, locally for the \'etale topology.
Moreover, as already noted in Equation \eqref{eq:degRR}, the existence of an RR-spin structure forces the relative degree $\nf_R$ of $Z$  to be even.

If the morphism $\pi\colon X\to S$ has a section, 
$\Pic X\to\Pic X/S$ is an epimorphism  \cite[Section 2]{GrPic62} and the `relative RR-spin structure' has
the form 
$\Upsilon=[\Lcl]$ for some invertible sheaf $\Lcl$ on $X$. However, the sheaf $\Lcl$ may 
still not be an RR-spin structure.
On the other hand, since $\pi_\ast\Oc_X\iso\Oc_S$, if there
is an RR-spin structure
 $\Lcl\in\Upsilon$ along $Z$,  any other RR-spin structure in the same class has the form
$\Lcl\otimes\pi^\ast \Nc$ for an invertible sheaf $\Nc$ on $S$ with
$\Nc^2=\Oc_S$.

\begin{defin}\label{def:RRfunct} The functor of RR-spin curves is the functor on the
category of ordinary schemes given by
\begin{equation}
S\to\F_{\nf_R}^{\,\mbox{\rm\tiny spin}}(S)=
\left\{
\begin{matrix}
\text{Isomorphism classes of RR-spin curves $(X/S,Z,\Upsilon)$}\\
\text{of genus $g$ along a positive divisor $Z$}\\
\text{of relative degree $\nf_R$ and nonramified over $S$}
\end{matrix}
\right\}\,.
\end{equation} 
\end{defin}

Proceeding as in \cite[Theorem 4.2]{DoHeSa97} we obtain:
\begin{prop}\label{p:RRspin}
There is
an isomorphism of sheaves for the \'etale topology
\begin{equation}\label{e:isofunctors}
\F_{\nf_R\vert \mathrm{\{Schemes\}}}\iso\F_{\nf_R}^{\,\mbox{\rm\tiny spin}};
\end{equation}
that is, the restriction of the functor of RR-SUSY curves to the category of ordinary
schemes is the functor of RR-spin curves.
\qed
\end{prop}

Consider now the morphism of functors over ordinary schemes $S$ given by
\begin{equation}\label{eq:functmorph}
\begin{aligned}
F_{\nf_R}(S)& \to \Div_g^{[\nf_R]}(S)=
\left\{
\begin{matrix}\hbox{Automorphism classes of curves}\cr 
               \hbox{$X\to S$ of genus $g$}\cr
               \hbox{with a positive divisor}\cr
               \hbox{$Z\hookrightarrow X$ of relative degree $\nf_R$}\cr
               \hbox{and nonramified over $S$}\end{matrix}
\right\} \\ \\
(\Xcal/S,\Zc,\Dcal) & \mapsto  (X/S, Z)\,.
\end{aligned}
\end{equation}

By Proposition \ref{p:RRspin}, 
the ``fibre functor'' is the functor of RR-spin structures on $X/S$ along $Z$.
Take the moduli scheme
$M_g$ of curves of genus $g$ and the universal curve $X_g\to M_g$. Let us consider the open subscheme $M$ of the relative symmetric power $S^{\nf_R}(X_g/M_g)$ which is the complementary of all the diagonals. Then $M$ is a fine moduli scheme for relative positive divisors of relative degree $\nf_R$ that are nonramified over the base (that is, whose fibres are positive divisors consisting of $\nf_R$ different points). This is a quasi-projective scheme of dimension $3g-3+\nf_R$. Let us denote by $X\to M$ the pullback of the universal curve, and by $Z\hookrightarrow X$ the universal positive divisor of $X\to M$ of relative degree $\nf_R$. Note that $Z\to M$ is nonramified.

Let $J^d=J^d(X/M)\to M$ be the
relative Jacobian of invertible sheaves of degree $d$ on $X_M\to M$ and
 $\Upsilon_d\in\Pic^d (X_M\times_M J^d/J^d)$ the universal ``invertible sheaf
class.''  The twisted canonical sheaf
$\kappa_{X_M/M}(Z)$ defines a point of $J^{2g-2+\nf_R}$ with values in
$M$, that is, a morphism $\iota\colon M\to J^{2g-2+\nf_R}$ of $M$-schemes. 
If $m=\nf_R/2$, $\mu_2\colon
J^{g-1+m}\to J^{2g-2+\nf_R}$ is the morphism ``raising to power two,'' and
$M_{\nf_R}=\mu_2^{-1}[\kappa_{X_M/M}(Z)]\subset J^{g-1+m}$ is the pullback of that
point, that is,  $M_{\nf_R}$ is defined by the cartesian diagram
$$
\xymatrix{
J^{g-1+m}\ar[r]^{\mu_2} & J^{2g-2+\nf_R} \\
M_{\nf_R}\ar@{^(->}[u]\ar[r]^{\rho}& M\ar@{^(->}[u]_\iota\,.
}
$$
The natural projection $M_{\nf_R}\to M$ is an \'etale covering of degree
$2^{2g}$, so that $M_{\nf_R}$ is a quasi-projective scheme of dimension $3g-3+\nf_R$.

Now, the pullback
$X_{\nf_R}=X_M\times_{M}M_{\nf_R}\hookrightarrow X_M\times_M J^{g-1+p}$ of the universal
curve, together with the pullback $Z_{\nf_R}$ of the divisor $Z$ and  the ``invertible sheaf class'' $\Upsilon_{\nf_R}$ given by the restriction of $\Upsilon_{g-1+p}$ to $X_{\nf_R}$, define a ``universal RR-spin curve class'' in  $\F_{\nf_R}^{\,\mbox{\rm\tiny spin}}(S)(M_{\nf_R})$,  as one has 
\begin{equation}\label{eq:univsquare}
\Upsilon_{\nf_R}^2=[\kappa_{X_{\nf_R}/M_{\nf_R}}(Z_{\nf_R})]. 
\end{equation}

\begin{thm}[Bosonic moduli]\label{thm:RRfunct} The functor 
of spin curves $\F_{\nf_R}^{\,\mbox{\rm\tiny spin}}(S)$ is representable by the quasi-projective scheme
$M_{\nf_R}$ of dimension $3g-3+\nf_R$, which we call the \emph{bosonic supermoduli scheme}, and the 'universal RR-spin curve class' 
$$
(X_{\nf_R}, Z_{\nf_R}, \Upsilon_{\nf_R})\,.
$$
Due to the functor isomorphism \eqref{e:isofunctors}, the restriction $\F_{\nf_R\vert \mathrm{\{Schemes\}}}$ of the functor of RR-SUSY curves to the category of ordinary schemes is representable by $M_{\nf_R}$ together with a ``universal RR-SUSY curve class''
$$
(\Xcal_{\nf_R}\to M_{\nf_R}, Z_{\nf_R}, {\Df}_{\nf_R})\,.
$$
\qed
\end{thm}

The ``universal relative RR-spin curve class''  $\Xcal_{\nf_R}\to M_{\nf_R}$ is not a true relative RR-spin curve with underlying scheme $X_{\nf_R}$, because $\Upsilon_{\nf_R}$ is not the class of a line bundle, however there exists an affine \'etale covering $p\colon U\to M_{\nf_R}$ such that $p^\ast\Upsilon_{\nf_R}$ is the class of a line bundle $\Lcl$ on the pull-back $X_U\to U$ of $X_{\nf_R}$ (for that it is enough that $X_U\to U$ has a section). By Equation \eqref{eq:univsquare}, $\Lcl\otimes\Lcl$ is isomorphic to $\kappa_{X_U/U}(Z_U)$, where $Z_U$ is the pull-back of the universal divisor $Z_{\nf_R}$, up to the pull-back of a line bundle on $U$. We can then refine $U$ if necessary to have an isomorphism
$$
\varpi\colon\Lcl\otimes\Lcl\iso \kappa_{X_U/U}(Z_U)\,.
$$

Analogously, the ``universal RR-SUSY curve class''
$(\Xcal_{\nf_R}\to M_{\nf_R}, Z_{\nf_R}, \Df_{\nf_R})$ is not a true relative RR-SUSY curve with underlying scheme $X_{\nf_R}$ because $\Df_{\nf R}$ is not an RR-superconformal structure. Actually, the ``universal RR-SUSY curve class'' is the section on $M_{\nf_R}$ of the sheaf $\F_{\nf_R}$ determined on the above affine \'etale covering  by the RR-SUSY curve $(X_U,Z_U,\Df_U)$, where $\Df_U$ is the RR-superconformal structure induced by $\varpi$ according to Proposition \ref{p:spinrel}.

\begin{defin}\label{def:trivializingRR}
An \'etale covering $p\colon U\to M_{\nf_R}$ is called \emph{trivializing} if it has the above properties, namely, $p^\ast\Upsilon_{\nf_R}=[\Lcl]$ for a line bundle $\Lcl$ on $X_U$ and there is an isomorphism $\varpi\colon\Lcl\otimes\Lcl\iso \kappa_{X_U/U}(Z_U)$.
\end{defin}

Take an affine \'etale covering $q\colon R\to U \times_{M_{\nf_R}}U$ and denote by $(q_1,q_2)\colon R \rra U$ the composition of $q$ with the two projections $(p_1,p_2)\colon U\times_{M_{\nf_R}}U \rra U$. Then $q_1^\ast X_U$ and $q_2^\ast X_U$ are isomorphic as curves over $R$; let us write $X_R:= q_1^\ast X_U\iso q_2^\ast X_U$. Also $q_1^\ast Z_U$ and $q_2^\ast Z_U$ are isomorphic as relative divisors over $R$ and we write $Z_R:= q_1^\ast Z_U\iso q_2^\ast Z_U$.
One then has a diagram
$$
\xymatrix{
X_R\ar@<3pt>[r]^{(q_1,q_2)}\ar@<-3pt>[r] \ar[d]_\pi& X_U\ar[r]\ar[d]_\pi & X_{\nf_R}\ar[d]_\pi
\\
R\ar@<3pt>[r]^{(q_1,q_2)}\ar@<-3pt>[r] & U \ar[r] &M_{\nf_R}
}
$$
where the double arrows are \'etale equivalence relations and the rightmost column contains the quotients in the category of schemes. We also have 
$$
q_1^\ast\varpi\colon q_1^\ast\Lcl\otimes q_1^\ast\Lcl \iso \kappa_{X_R/R}(Z_R)\,,\quad q_2^\ast\varpi\colon q_2^\ast\Lcl\otimes q_2^\ast\Lcl \iso \kappa_{X_R/R}(Z_R)\,.
$$
The line bundles $q_1^\ast\Lcl$ and $q_2^\ast\Lcl$ are in the same class in the relative Picard group, so that $q_1^\ast\Lcl\iso q_2^\ast\Lcl\otimes \pi^\ast\Nc$ for a line bundle $\Nc$ on the base.  By refining $q$ if necessary, we can assume that $\Nc$ is trivial; then $q_1^\ast\Lcl\iso q_2^\ast\Lcl$. Let us denote by $\bar\Lcl$ one of these line bundles. Proceeding as in the proof of  \cite[Theorem 4.2]{DoHeSa97}, we can refine $q$ again so that there is an automorphism $t$ of $\bar\Lcl$ such that 
$q_2^\ast\varpi=q_1^\ast\varpi\circ(t\otimes t)$. 

Thus, considering the relative RR-SUSY curves
$\Xcal_R=(X_R, \Oc_{X_R}\oplus \bar\Lcl)$ and $ \Xcal_U=(X_U, \Oc_{X_U}\oplus \Lcl)$, 
the above diagram can be completed to a diagram of RR-SUSY curve classes
\begin{equation}\label{eq:RRusy}
\xymatrix{
(\Xcal_R,Z_R,\bar\Lcl)\ar@<3pt>[r]^{(q_1,q_2)}\ar@<-3pt>[r] \ar[d]_{\pi_R}& 
(\Xcal_U,Z_U,\Lcl)\ar[r]\ar[d]_{\pi_U} & (\Xcal_{\nf_R},Z_{\nf_R},\Upsilon_{\nf_R})\ar[d]_\pi
\\
R\ar@<3pt>[r]^{(q_1,q_2)}\ar@<-3pt>[r] & U \ar[r] &M_{\nf_R}
}
\end{equation}
where we are confusing superschemes with their functors of   points. 
Moreover, since $X_{\nf_R}$ is separated, the morphism $X_R\xrightarrow{(q_1,q_2)} X_U\times X_U$ is a closed immersion. It follows that $\Xcal_R\xrightarrow{(q_1,q_2)} \Xcal_U\times \Xcal_U$ is a closed immersion as well.
This proves the following:
\begin{prop}\label{p:univcurvebos}
The ``universal relative RR-SUSY curve class'' is the quotient of the \'etale equivalence of superschemes
$$
\Xcal_R\rra \Xcal_U\to \Xcal_{\nf_R}\,.
$$
That is, $\Xcal_{\nf_R}$ is an Artin algebraic superspace (see \cite{DoHeSa97} Definition 6.3) whose underlying Artin algebraic space is the relative curve $X_{\nf_R}\to M_{\nf_R}$ over the bosonic moduli scheme. Moreover $\Xcal_{\nf_R}$ is separated \cite[Chap.~2, 1.8]{Knut71}.
\qed
\end{prop}

Since $U$ and $R$ are local moduli spaces for RR-SUSY curves on the category of ordinary schemes, one has:
 
\begin{prop}\label{p:ksisoRR} The even components  of the Kodaira-Spencer maps (Definition \ref{ksmap3}) of the RR-SUSY curves $(\Xcal_U,Z_U,\Lcl)\to U$ and $(\Xcal_R,Z_R,\bar\Lcl)\to R$ are isomorphisms.
\qed\end{prop}

\begin{remark}\label{rem:trivcov} For every trivializing covering $p\colon
U\to M_{\nf_R}$ (Definition \ref{def:trivializingRR}), the fibre product $R=U\times_{M_{\nf_R}}U$ defines an \'etale
equivalence relation
$(p_1,p_2)\colon U\times_{M_{\nf_R}}U\rra U$, whose quotient (in the category
of locally ringed spaces, \cite{Knut71}) is
$M_{\nf_R}$. Then, the functor of RR-spin curves
$\F_{\nf_R}\simeq M_{\nf_R}$ is the quotient of equivalence relation
$$
R\simeq{U}\times_{M_{\nf_R}}
{U}\rra {U}
$$ in the category of sheaves of sets.
\end{remark}

\subsection{RR-SUSY curves over split superschemes and extension problems}

Let $\Sc=(S, \bigwedge \Ec)$ be an \emph{affine} split superscheme and $(\pi\colon\Xcal_\ast\to S, Z, \Lcl)$ an RR-SUSY curve over the underlying ordinary scheme $S$. Then $\Oc_{\Xcal_\ast}=\Oc_X\oplus\Lcl$ and there is an isomorphism $\Lcl\otimes\Lcl\iso\kappa_{X/S}(Z)$.

We want to describe the supercurves $\bar\pi\colon\Xcal\to\Sc$ with an RR-SUSY structure $\Dcal$ along a superdivisor $\Zc$ extending $(\Xcal_\ast, Z, \Lcl)$. We consider two such supercurves $\Xcal'\to\Sc$ and $\Xcal\to \Sc$ as equivalent if there exists an isomorphism $\Xcal'\iso\Xcal$ of relative RR-SUSY curves over $\Sc$ which restrict to the identity on $(\Xcal_\ast, Z, \Lcl)$. 

In this case we consider the sheaf of abelian groups over $X$ given by
$$
G=\mathrm{exp}(\Nc\otimes\Gc_\pi^{\nf_R})_0\,,
$$
where $\Gc_\pi^{\nf_R}$ is the sheaf 
given by Definition \ref{def:infdefRR}, $\Nc$ is the ideal of $\bigwedge \Ec$ generated by $\Ec$ and the subscript $0$ denotes the even part, and $\Nc\otimes\Gc_\pi^{\nf_R}:=\pi^{-1}\Nc\otimes_{\pi^{-1}\Oc_S}\Gc_\pi^{\nf_R}$.

Since any RR-SUY curve is locally split on the source by Proposition \ref{p:localsplitRR} and $\pi_\ast(\Gc_\pi^{\nf_R})=0$ by Corollary \ref{cor:h1dimrel}, similar statements to Lemmas 2.3 and 2.4 of \cite{LeRoth88} hold true with analogous proofs:

\begin{lemma}\label{lem:LBR1} There is a one-to-one correspondence between the set of equivalence classes of RR-SUSY curves $(\Xcal\to\Sc,\Zc,\Dcal)$ extending  $(\Xcal_\ast\to S, Z, \Lcl)$ and the cohomology set  $H^1(X,G)$.
\qed
\end{lemma}

\begin{lemma} \label{lem:LBR2} 
If $\Sc^{(m)}=(S,\Oc_\Sc/\Nc^{m+1})$, 
then
\begin{enumerate}
\item
any RR-SUSY curve $(\pi^{(m)}\colon\Xcal \to \Sc^{(m)},\Zc^{(m)},\Dcal^{(m)})$ can be extended to an RR-SUSY curve over $\Sc^{(m+1)}$;
\item  the space of equivalence classes of such extensions is naturally an affine space modelled on
$$
\Gamma(S, (\textstyle{\bigwedge}^m\Ec\otimes R^1\pi_\ast\Gc_\pi^{\nf_R})_0)\,.
$$
\end{enumerate}
\qed
\end{lemma}

\subsection{Local structure of the moduli of RR-SUSY curves}\label{ss:localstructure}

The local structure of the supermoduli of RR-SUSY curves $\M_{\nf_R}$, assuming   that the latter exists, can be determined by looking at the infinitesimal deformations of the universal RR-SUSY curve class  $\Xcal_{\nf_R}\to M_{\nf_R}$ over the bosonic moduli $M_{\nf_R}$ given by Theorem \ref{thm:RRfunct}. 
Since the underlying scheme to $\M_{\nf_R}$ is $M_{\nf_R}$, one should have locally that 
$$
\M_{\nf_R}=(M_{\nf_R}, \bigwedge \Ec)\,,
$$
where $\Ec$ is a locally free sheaf on $M_{\nf_R}$ determined by 
$$
\Ec^\ast \simeq \Theta(\M_{\nf_R})_-^\Pi
$$
(see Equation \eqref{eq:bostangent}). If $\M_{\nf_R}$ exists,  $\Theta_{\M_{\nf_R}\vert M_{\nf_R}}$ is identified by Proposition \ref{p:tangentdefor} with the sheaf of
the (local) elements in
$\Hom(M_{\nf_R}[\epsilon_0,\epsilon_1],\M_{\nf_R})$ that are extensions of the immersion $M_{\nf_R}\hookrightarrow \M_{\nf_R}$. Even if we do not yet  know   whether $\M_{\nf_R}$ exists, we can make sense of the homomorphisms of any superscheme $\Sc$ to it: they are the elements of $\F_{\nf_R} (\Sc)$, that is, the relative RR-SUSY curves classes over $\Sc$.
It follows that  $\Ec^\ast$ \emph{is the sheaf of the odd infinitesimal deformations of the universal RR-SUSY curve class} $(\Xcal_{\nf_R}\to M_{\nf_R},\Zc_{\nf_R},\Upsilon_{\nf_R})$ (sse Proposition \ref{relksmap1}). 

We study those infinitesimal deformations by computing them locally for the \'etale topology.
Take a  trivializing cover  $p\colon U\to M_{\nf_R}$ (Definition \ref{def:trivializingRR}), so that 
$p^\ast\Upsilon_{\nf_R}=[\Lcl]$ for a line bundle $\Lcl$ on $X_U$ and there is an isomorphism $\varpi\colon\Lcl\otimes\Lcl\iso \kappa_{X_U/U}(Z_U)$. Let us write $\Lcl=\kappa_{X_U/U}(Z_U)^{1/2}$.
The sheaf of infinitesimal deformations is given by Corollary \ref{cor:h1dimrel}. We then have that the pull-back of $\Ec$ to $U$ should be given by
$$
\Ec_U\simeq (R^1\pi_{U\ast}(\kappa_{X_U/U}(Z_U)^{-1/2}))^\ast\iso \pi_{U\ast} (\kappa_{X_U/U}(Z_U)^{1/2}\otimes\kappa_{X_U/U})\,,
$$
where the second equality if relative Serre duality.
Then  the candidate for a local supermoduli of RR-SUSY curves is:

\begin{defin}\label{def:localmoduliRR} 
The local supermoduli of RR-SUSY curves is the superscheme:
$$
\U=(U, \bigwedge \Ec_U)\,,
$$
where $\Ec_U\simeq  \pi_{U\ast} (\kappa_{X_U/U}(Z_U)^{1/2}\otimes\kappa_{X_U/U})\simeq \pi_{U\ast} (\Lcl\otimes \kappa_{X_U/U})$. One has:
\begin{equation}\label{eq:dimRR}
\dim\U=(3g-3+\nf_R,2g-2+\nf_R/2)\,.
\end{equation}
\end{defin}

Proceeding again as in \cite{LeRoth88} (see also  \cite{DoHeSa97}) one has:

\begin{thm}\label{LBR1-RR} Let $(\pi\colon \Xcal\to V, Z,\Lcl)$ be an RR-SUSY curve over an affine scheme whose even
Kodaira-Spencer map $ks_0(\pi)$ is an isomorphism. Consider the sheaf $\Ec=\pi_\ast (\kappa_{X/V}(Z)^{1/2}\otimes\kappa_{X/V})$. Then, there is an RR-SUSY curve $(\bar\pi\colon \Xf\to \Vc, \Zc, \Dcal)$
over the superscheme $\Vc=(V,
\bigwedge \Ec)$ extending
$\Xcal$, whose
Kodaira-Spencer map $ks(\bar\pi)$ is an isomorphism.
\qed
\end{thm}

\begin{thm}\label{LBR2-RR} Let $(\Xf\to \Vc,\Zc{,} \Dcal)$ be a
relative RR-SUSY curve whose Kodaira-Spencer map is an isomorphism and let 
$(\Xcal \to V,Z,\Lcl)$ the underlying relative RR-SUSY curve over the ordinary scheme $V$.
For every morphism of schemes  $\varphi\colon
S\to V$, there is a one-to-one correspondence
\begin{equation}
\begin{aligned}
\left\{
\begin{matrix}
		\text{morphisms of superschemes}\\ 
               \text{$\Sc\to\Vc$ extending}\\
               \varphi\colon S\to V
\end{matrix}
\right\} & \to 
\left\{
\begin{matrix}
		\hbox{classes of relative RR-SUSY curves}\cr 
               \hbox{$(\bar\Xf/\Sc,\bar\Zc,\bar\Dcal)$ extending}\cr
             \varphi^\ast(\Xcal \to V,Z,\Lcl)
               \end{matrix}
\right\} 
\\
\phi & \mapsto  \phi^\ast\Xf/\Sc\,.
\end{aligned}
\end{equation}
\qed\end{thm}

\subsection{Global moduli of RR-SUSY curves and universal RR-SUSY curves}

In this subsection we give a construction of the moduli space of RR-SUSY curves as an (Artin) algebraic superspace (Definition \ref{def:etalerel}).
 
In view of Theorems \ref{LBR1-RR} and \ref{LBR2-RR}, the construction of the supermoduli of SUSY curves given in  \cite{DoHeSa97} can be now adapted, with similar proofs, to the case of RR-SUSY curves.
We consider the local supermoduli of RR-SUSY curves
$$
\U=(U, \bigwedge \Ec_U)\,,
$$
given by Definition \ref{def:localmoduliRR}.
Since the even component  of the Kodaira-Spencer map of the relative RR-SUSY curve $(\Xcal_U,Z_U,\Lcl)$ over $U$ is an isomorphism by Proposition \ref{p:ksisoRR}, we can apply Theorem \ref{LBR1-RR} to obtain a relative RR-SUSY curve
$$
(\Xf_\U,\Zc_U,\Dcal_U)\to \U,
$$
extending $(\Xcal_U,Z_U,\Lcl)\to U$ and
whose Kodaira-Spencer map is an isomorphism.
There is a natural morphism of presheaves $U\to SC_{\nf_R\vert \mathrm{\{Schemes\}}}$ which maps a scheme morphism $\varphi\colon S\to U$ to the RR-SUSY curve $\varphi^\ast \Xcal_{U}\to S$ (recall that $F_{\nf_R}$ is the \emph{presheaf} of relative RR-SUSY curves).
We can then consider the fibre product $SC_g^m\times_{S_{\nf_R\vert \mathrm{\{Schemes\}}}} U$ as a presheaf of 
superschemes. A section of this presheaf on a supercheme $\Sc$
is a  pair formed by an RR-SUSY curve
$\Xf\to\Sc$ and a morphism of schemes $\varphi\colon S\to U$,
such that $\Xf\to\Sc$ is an extension of $\varphi^\ast \Xcal_{U}\to S$.
By Theorem \ref{thm:RRfunct}, the sheaf associated to the presheaf $F_{\nf_R}\times_{S_{\nf_R\vert \mathrm{\{Schemes\}}}} U$ is
$$
\F_{\nf_R}\times_{M_{\nf_R}} U\,,
$$ 
where $\F_{\nf_R}$ is the sheaf of relative RR-SUSY curves.

\begin{lemma}\label{lem:pointsfunctorRR} There is an isomorphism of sheaves of superschemes
$$
\varpi\colon\U \iso \F_{\nf_R}\times_{M_{\nf_R}} U
\,,
$$
 which maps a superscheme morphism $\phi\colon \Sc \to \U$ to $\phi^\ast (\Xf_U,\Zc_U,\Dcal_U), \varphi)$ where $\varphi\colon S\to U$ is the underlying scheme morphism to $\phi$.\qed
\end{lemma}
\begin{proof}
By Theorem~\ref{LBR2-RR}, there is a presheaf isomorphism
\begin{equation}
\U\iso F_{\nf_R}\times_{S_{\nf_R\vert \mathrm{\{Schemes\}}}} U\,.
\end{equation}
Since (the functor of   points of) $\U$ is a sheaf, the presheaf of the right-hand side is also a sheaf, so that it coincides with its associated sheaf $\F_{\nf_R}\times_{M_{\nf_R}} U$.
\end{proof}

The trivial \'etale equivalence relation $(p_1,p_2)\colon U\times_{M_{\nf_R}}U \rra U$ induces an equivalence relation of 
sheaves of superschemes
$$
\xymatrix{
\F_{\nf_R}\times_{M_{\nf_R}}  U\times_{M_{\nf_R}}U\ar@<3pt>[r]^(.6){(p_1,p_2)}\ar@<-3pt>[r]& \F_{\nf_R} \times_{M_{\nf_R}} U\ar[r]^(.6)p& \F_{\nf_R}\,
 }
$$
with categorial quotient $\F_{\nf_R}$.
Since there is a categorical quotient $U\times_{M_{\nf_R}}U\rra U\to M_{\nf_R}$ in the category of schemes, one can apply Lemma \ref{lem:pointsfunctorRR} and Equation \eqref{eq:dimRR} to prove:

\begin{thm}\label{thm:globlamoduiRR} The functor $\F_{\nf_R}$ of relative RR-SUSY curves of genus $g$ along a (nonramifiec) puncture of degree $\nf_R$ is the categorial quotient
$$
\xymatrix{
\U\times_{M_{\nf_R}}U\ar@<3pt>[r]^(.65){(\rho_1,\rho_2)}\ar@<-3pt>[r]& \U\ar[r]^\rho& \F_{\nf_R}\,,
 }
$$
in the category of sheaves of superschemes, where $\rho_i=\varpi^{-1}\circ p_i\circ (\varpi\times 1)$ and $\rho=p\circ\varpi$. Moreover, $(\rho_1,\rho_2)\colon \U\times_{M_{\nf_R}}U \rra \U$ is an \'etale equivalence relation of superschemes, so that $\F_{\nf_R}$ is identified with a separated Artin algebraic superspace $\M_{\nf_R}$ whose ordinary underlying algebraic space is the scheme $M_{\nf_R}$. Moreover,
$$
\dim \M_{\nf_R}=(3g-3+\nf_R,2g-2+\nf_R/2)\,.
$$
\qed
\end{thm}

The separateness is proved as in Proposition \ref{p:univcurvebos}.

\subsection{Universal relative RR-SUSY curves}

 We now construct universal relative RR-SUSY curve classes on the supermodui $\M_{\nf_R}$.
 The algebraic superspace morphism $\rho \colon\U\to \M_{\nf_R}$ is induced by the relative RR-SUSY curve $\Xf_{\U}\to \U$, so that one has ${\Xf_{\U}}\iso\rho^\ast \Xf_{\nf_R}$ where $\Xf_{\nf_R}$ is the universal element corresponding to the identity in $\Sc\Cc_g^m(\M_{\nf_R})=\Hom(\M_{\nf_R},\M_{\nf_R})$.

If we consider the \'etale equivalence relation of superschemes $(\rho_1,\rho_2)\colon \U\times_{M_{\nf_R}}U \rra \U$, one has $\rho\circ\rho_1=\rho\circ\rho_2$, and then, $\rho_1^\ast {\Xf_{\U}}\iso \rho_2^\ast {\Xf_{\U}}$ as relative RR-SUSY curve classes. It follows that there is an \'etale covering of superschemes $\psi\colon \Rcal\to \U\times_{M_{\nf_R}}U$ such that $\bar\rho_1^\ast \Xf_{\U}\iso \bar\rho_2^\ast \Xf_{\U}$ as \emph{true relative RR-SUSY curves} over $\Rcal$, where we have written $\bar\rho_1=\rho_1\circ\psi$, $\bar\rho_2=\rho_2\circ\psi$. Let us denote by $\Xf_\Rcal$ any of them. The pull-back of $(\bar\rho_1,\bar\rho_2)$ by $\Xf_{\U}\to \U$ gives an \'etale equivalence relation of superschemes $(q_1,q_2)\colon \Xf_\Rcal  \rra \Xf_{\U}$
and a diagram 

\begin{equation}\label{eq:univsupercurveRR}
\xymatrix{
(\Xf_{\Rcal},\Zc_\Rcal,\Dcal_\Rcal) \ar@<3pt>[r]^{(q_1,q_2)}\ar@<-3pt>[r] \ar[d]& (\Xf_\U,\Zc_U,\Dcal_U) \ar[r] \ar[d]&(\Xf_{\nf_R},\Zf_{\nf_R},\Df_{\nf_R}) \\
\Rcal \ar@<3pt>[r]^{(\bar\rho_1,\bar\rho_2)}\ar@<-3pt>[r] & \U \ar[r] & \M_{\nf_R}
}
\end{equation}

\begin{prop}\label{p:universalcurveRR}
The sheaf $\Xf_{\nf_R}$ is representable by the quotient of an \'etale equivalence relation of superschemes, that is, is representable by a separated (Artin) algebraic superspace. Moreover, the two vertical morphisms in the diagram \eqref{eq:univsupercurveRR} induce a morphism of algebraic superspaces
$$
\Xf_{\nf_R}\to\M_{\nf_R}\,,
$$
which endows $(\Xf_{\nf_R}\to\M_{\nf_R}, \Zf_{\nf_R},\Df_{\nf_R})$ with the structure of a relative RR-SUSY curve class. We call it the \emph{universal relative RR-SUSY curve} of genus $g$ and degree $\nf_R$.
\qed
\end{prop}

\bigskip
\section{Supermoduli of RR-SUSY curves with NS punctures}\label{s:allpunctures}

This last section is devoted to the construction a supermoduli for SUSY curves with both Neveu-Schwarz (NS) and \RR\  (RR) punctures. The construction follows directly  from the construction of the supermoduli of SUSY curves with NS punctures given in \cite{DoHeSa97} and the above construction of the supermoduli for RR-SUSY curves.

Recall that an NS puncture of degree $\nf_{NS}$ on a relative SUSY curve $(\pi\colon\Xcal\to\Sc, \Dcal)$ is an ordered family $(s_1,\dots,s_{\nf_{NS}})$ of  sections $s\colon \Sc\hookrightarrow \Xcal$ of $\pi$ and that a morphism of SUSY curves with NS  punctures is a morphism of SUSY curves that preserve the corresponding sections.
If $\pi\colon\Xcal\to \Sc$ is a supercurve, one can define the (relative) $N$-supersymmetric power $\Xcal^{[N]}\to\Sc$ as the quotient superscheme of the cartesian product $\Xcal^N$ by the natural action of the symmetric group $S_N$ by automorphisms of superschemes \cite{DoHeSa91,DoHeSa93} (in these references is proven that, $\Xcal^{[N]}$ is locally split if and only if the relative odd dimension is $n=1$).
Then, for any supercurve $\pi\colon\Xcal\to \Sc$ of relative dimension $(1,1)$, the $N$-supersymmetric power is a locally split superscheme and the morphism $\Xcal^{[N]}\to\Sc$ has relative dimension $(N,N)$.

Let us fix the number $\nf_R$ of \RR\ punctures and the number $\nf_{NS}$ of Neveu-Schwarz punctures,
and let us consider the diagram  \eqref{eq:univsupercurveRR}. Taking $\nf_{NS}$ symmetric powers on the first two columns, which are relative supercurves, we get a new diagram
$$
\xymatrix{
\Xf_{\Rcal}^{[\nf_{NS}]}  \ar@<3pt>[r]^{(q_1,q_2)}\ar@<-3pt>[r] \ar[d]& \Xf_\U^{[\nf_{NS}]}  \ar[r] \ar[d]& \Xf_{\nf_{NS},\nf_{R}} \ar[d] \\
\Rcal \ar@<3pt>[r]^{(\bar\rho_1,\bar\rho_2)}\ar@<-3pt>[r] & \U \ar[r] & \M_{\nf_R}
}
$$
where $\Xf_{\nf_{NS},\nf_{R}}$ is the algebraic superspace defined as the quotient of the \'etale equivalence relation $(q_1,q_2)$.

 Combining Theorem \ref{thm:globlamoduiRR} and \cite[Theorem 31]{DoHeSa97} we have:
 
 \begin{thm}\label{thm:mod}  The functor of relative RR-SUSY curves of genus $g$ along a (nonramified) RR puncture of degree $\nf_R$   with $\nf_{NS}$ NS punctures is representable by a separated algebraic superspace $ \Xf_{\nf_{NS},\nf_{R}}$, which is the $\nf_{NS}$-supersymmetric power of the relative universal RR-SUSY curve class $\Xf_{\nf_R}\to\M_{\nf_R}$ over the supermoduli of RR-SUSY curves along an RR puncture of degree $\nf_R$. Its dimension is
\begin{equation}\label{eq:dimRR}
 \dim \Xf_{\nf_{NS},\nf_{R}} = (3g-3+\nf_{NS}+\nf_R, 2g-2+\nf_{NS}+\nf_R/2)\,.
\end{equation}
\qed \end{thm}

We have assumed that the positive superdivisors on which the \RR\ punctures are supported are reduced, or, in the case of families, that the superdivisors are unramified over the base; the reason is that  Proposition \ref{p:localsplitRR} or the local structure of the supermoduli as given in Subsection \ref{ss:localstructure} may fail to be true without that assumption. This is not needed in the  case NS-punctures, as we have seen in Theorem \ref{thm:mod}. In string theory punctures are   insertion points of vertex operators, so that to avoid divergences usually the assumption is made that the sections defining an NS puncture of degree $\nf_{NS}$ determine $\nf_{NS}$ different points on every fibre, and moreover that these points do not belong to the positive superdivisors along which the \RR\ punctures are supported \cite[Sections 6.2.2, 6.2.3]{Witten19}.

So one might like to have a construction where the punctures are supposed to be distinct, in the sense specified above. This can be easily achieved; given a supercurve $\pi\colon\Xcal\to \Sc$, and a positive relative superdivisor $\Zc\hookrightarrow X$,   write $\Ycal= \Xcal-\Zc$. If $\Ycal^N_0$ is the open subsuperscheme of the cartesian product $\Ycal^N$ that parameterizes families of $N$ different points (that is, the complementary of all the partial diagonals), denote by $\Ycal^{(N)}\to \Sc$ the quotient superscheme of  $\Ycal^N_0$ by the natural action of the symmetric group; it is an open subsuperscheme of both $\Ycal^{[N]}$ and $\Xcal^{[N]}$.
Proceeding as before, we get a diagram
$$
\xymatrix{
\Yf_{\Rcal}^{(\nf_{NS})}  \ar@<3pt>[r]^{(q_1,q_2)}\ar@<-3pt>[r] \ar[d]& \Yf_\U^{(\nf_{NS})}  \ar[r] \ar[d]& \Yf_{\nf_{NS},\nf_{R}} \ar[d] \\
\Rcal \ar@<3pt>[r]^{(\bar\rho_1,\bar\rho_2)}\ar@<-3pt>[r] & \U \ar[r] & \M_{\nf_R}
}
$$
where $\Yf_\U=\Xf_U-\Zc_\U$, $\Yf_\Rcal=\Xf_U-\Zc_\Rcal$ and $\Yf_{\nf_{NS},\nf_{R}}$ is the algebraic superspace defined as the quotient of the \'etale equivalence relation $(q_1,q_2)$.

This gives the moduli we are seeking for:
   \begin{thm}\label{thm:mod2}  The functor of relative RR-SUSY curves of genus $g$ along a (nonramified) RR puncture of degree $\nf_R$   with $\nf_{NS}$  fibrewise distinct NS punctures that do not collide with the RR-punctures is representable by a separated algebraic superspace $ \Yf_{\nf_{NS},\nf_{R}}\to\M_{\nf_R}$, which is an open algebraic sib-superspace of $\Xf_{\nf_R}$.\qed \end{thm}
  The dimension of this moduli space is again that given in \eqref{eq:dimRR}, of course. 
 


\bigskip\frenchspacing
\def\cprime{$'$}

\end{document}